\documentclass[12pt,leqno]{article}
\usepackage{amsfonts,amsthm,amsmath}

\usepackage[mathscr]{eucal}
\usepackage{amsmath}
\usepackage{amsfonts}
\usepackage{amssymb}
\usepackage{footnpag}
\usepackage[dvips]{graphicx}

\usepackage[mathlines]{lineno}

\theoremstyle{plain}
\newtheorem{thm}{Theorem}[section]
\newtheorem{prop}{Proposition}[section]
\newtheorem{lem}{Lemma}[section]
\newtheorem{cor}{Corollary}[section]
\newtheorem{conj}{Conjecture}[section]

\theoremstyle{definition}
\newtheorem{df}{Definition}[section]
\newtheorem{rem}{Remark}[section]
\newtheorem{ex}{Example}[section]

\newcommand{\Q}{\mathbb{Q}}
\newcommand{\Z}{\mathbb{Z}}

\newcommand{\C}{\mathbb{C}}
\newcommand{\R}{\mathbb{R}}

\newcommand{\OO}{\mathcal{O}}
\newcommand{\oo}{\mathfrak{o}}

\newcommand{\QQ}{\mathbb{Q}}

\newcommand{\ZZ}{\mathbb{Z}}
\newcommand{\RR}{\mathbb{R}}

\newcommand{\NN}{\mathbb{N}}

\newcommand{\MC}{\multicolumn}

\newcommand{\Cl}{\mathop{\mathrm{Cl}}\nolimits}

\newcommand{\e}{\hfill $\Box$}

\begin{document}

\title{Toy models for D. H. Lehmer's conjecture II}
\author{Eiichi Bannai\thanks{Faculty of Mathematics, 
Kyushu University, 
Motooka 744 Nishi-ku, 
Fukuoka, 819-0395 Japan, 
bannai@math.kyushu-u.ac.jp}
and 
Tsuyoshi Miezaki\thanks{
Division of Mathematics, 
Graduate School of Information Sciences, 
Tohoku University, 
6-3-09 Aramaki-Aza-Aoba, Aoba-ku, Sendai 980-8579, Japan, 
miezaki@math.is.tohoku.ac.jp}
}
\date{}

\maketitle



\begin{quote}
{\small\bfseries Abstract.}

In the previous paper, we studied the ``Toy models 
for D. H. Lehmer's conjecture". Namely, we showed 
that the 
$m$-th Fourier coefficient of the weighted theta series 
of the $\mathbb{Z}^2$-lattice and the $A_{2}$-lattice 
does not vanish, when the shell of norm $m$ of those lattices 
is not the empty set. In other words, the spherical $4$ (resp. $6$)-design 
does not exist among the nonempty shells 
in the $\mathbb{Z}^2$-lattice (resp. $A_{2}$-lattice).

This paper is the sequel to the previous paper. We take $2$-dimensional 
lattices associated to the 
algebraic integers of 
imaginary quadratic fields whose class number is either $1$ or $2$, 
except for $\QQ(\sqrt{-1})$ and $\QQ(\sqrt{-3})$, 
then, show that 
the $m$-th Fourier coefficient of the weighted theta series 
of those lattices does not vanish, 
when the shell of norm $m$ of those lattices is not the empty set. 
Equivalently, we show that the corresponding spherical $2$-design 
does not exist among the nonempty shells in those lattices. 

\vspace{-0.15in}

\quad
\end{quote}

\section{Introduction}                                   
The concept of spherical $t$-design is due to Delsarte-Goethals-Seidel \cite{DGS}. For a positive integer $t$, a finite nonempty subset $X$ of the unit sphere
\[
S^{n-1} = \{x = (x_1, x_2, \cdots , x_n) \in \R ^{n}\ |\ x_1^{2}+ x_2^2+ \cdots + x_n^{2} = 1\}
\]
is called a spherical $t$-design on $S^{n-1}$ if the following condition is satisfied:
\[
\frac{1}{|X|}\sum_{x\in X}f(x)=\frac{1}{|S^{n-1}|}\int_{S^{n-1}}f(x)d\sigma (x), 
\]
for all polynomials $f(x) = f(x_1, x_2, \cdots ,x_n)$ of degree not exceeding $t$. Here, the righthand side means the surface integral on the sphere, and $|S^{n-1}|$ denotes the surface volume of the sphere $S^{n-1}$. The meaning of spherical $t$-design is that the average value of the integral of any polynomial of degree up to $t$ on the sphere is replaced by the average value at a finite set on the sphere. A finite subset $X$ in $S^{n-1}(r)$, the sphere of radius $r$ 
centered at the origin, 
is also called a spherical $t$-design if $\frac{1}{r}X$ is 
a spherical $t$-design on the unit sphere $S^{n-1}$.

We denote by ${\rm {\rm Harm}}_{j}(\R^{n})$ the set of homogeneous harmonic polynomials of degree $j$ on $\R^{n}$. It is well known that $X$ is a spherical $t$-design if and only if the condition 
\begin{eqnarray*}
\sum_{x\in X}P(x)=0 
\end{eqnarray*}
holds for all $P\in {\rm Harm}_{j}(\R^{n})$ with $1\leq j\leq t$. If the set $X$ is antipodal, that is $-X=X$, and $j$ is odd, then the above condition is fulfilled automatically. So we reformulate the condition of spherical $t$-design on the antipodal set as follows:
\begin{prop}
A nonempty finite antipodal subset $X\subset S^{n-1}$ is a spherical $2s+1$-design if the condition 
\begin{eqnarray*}
\sum_{x\in X}P(x)=0 
\end{eqnarray*}
holds for all $P\in {\rm Harm}_{2j}(\R^{n})$ with $2\leq 2j\leq 2s$. 
\end{prop}
It is known \cite{DGS} that there is a natural lower bound (Fisher type inequality) 
for the size of a spherical $t$-design in $S^{n-1}$. 
Namely, if $X$ is a spherical $t$-design in $S^{n-1}$, then 
\begin{eqnarray*}
\vert X\vert\geq \binom{n-1+[t/2]}{[t/2]}+\binom{n+[t/2]-2}{[t/2]-1} 
\end{eqnarray*}
if $t$ is even, and 
\begin{eqnarray}
\vert X\vert\geq 2\binom{n-1+[t/2]}{[t/2]} \label{ine:Fisher} 
\end{eqnarray}
if $t$ is odd. 

A lattice in $\R^{n}$ is a subset $\Lambda \subset \R^{n}$ with the property that there exists a basis $\{v_{1}, \cdots, v_{n}\}$ of $\R^{n}$ such that $\Lambda =\Z v_{1}\oplus \cdots \oplus\Z v_{n}$, i.e., $\Lambda $ consists of all integral linear combinations of the vectors $v_{1}, \cdots, v_{n}$. 
The dual lattice $\Lambda$ is the lattice
\begin{eqnarray*}
\Lambda^{\sharp}:=\{y\in \R^{n}\ |\ (y, x) \in\Z , \ \mbox{for all } x\in \Lambda\}, 
\end{eqnarray*}
where $(x, y)$ is the standard Euclidean inner product. 
The lattice $\Lambda $ is called integral if $(x, y) \in\Z$ for all $x$, $y\in \Lambda$. An integral lattice is called even if $(x, x) \in 2\Z$ for all $x\in \Lambda$, and it is odd otherwise. An integral lattice is called unimodular if $\Lambda^{\sharp}=\Lambda$. 
For a lattice $\Lambda$ and a positive real number $m>0$, the shell of norm $m$ of $\Lambda$ is defined by 
\[
\Lambda_{m}:=\{x\in \Lambda\ |\ (x, x)=m \}=\Lambda\cap S^{n-1}(m).
\]

Let $\mathbb{H} :=\{z\in\C\ |\ {\rm Im}\,(z) >0\}$ be the upper half-plane. 
\begin{df}
Let $\Lambda$ be the lattice of $\R^{n}$. Then for a polynomial $P$, the function 
\begin{eqnarray*}
\Theta _{\Lambda, P} (z):=\sum_{x\in \Lambda}P(x)e^{i\pi z(x, x)}
\end{eqnarray*}
is called the theta series of $\Lambda $ weighted by $P$. 
\end{df}
\begin{rem}[See Hecke \cite{Hecke}, Schoeneberg \cite{{Sch1},{Sch2}}]
$\\ $
{\rm (i)}
When $P=1$, we get the classical theta series 
\begin{eqnarray*}
\Theta _{\Lambda} (z)=\Theta _{\Lambda, 1} (z)=\sum_{m\ge 0}|\Lambda_{m}|q^{m},\ {\rm where}\ q=e^{\pi i z}. 
\end{eqnarray*}
{\rm (ii)}
The weighted theta series can be written as 
\begin{eqnarray*}
\Theta _{\Lambda, P} (z)&=&\sum_{x\in \Lambda}P(x)e^{i\pi z(x, x)} \\
&=&\sum_{m\geq 0}a^{(P)}_{m}q^{m},\ {\rm where}\ a^{(P)}_{m}:=\sum_{x\in \Lambda_{m}}P(x). 
\end{eqnarray*}
\end{rem}

These weighted theta series have been used efficiently for the study of spherical designs which are the nonempty shells of Euclidean lattices. (See \cite{{Venkov1},
{Venkov2},{HP},{Pache},{HPV}}. See also \cite{BKST}.) 

\begin{lem}[cf.~\cite{{Venkov1},{Venkov2}},~{\cite[Lemma 5]{Pache}}]\label{lem:lempache}
Let $\Lambda$ be an integral lattice in $\R^{n}$. Then, for $m>0$, the non-empty shell $\Lambda_{m}$ is a spherical $t$-design if and only if 
\begin{eqnarray*}
a^{(P)}_{m}=0
\end{eqnarray*}
for all $P\in {\rm Harm}_{2j}(\R^{n})$ with $1\leq 2j\leq t$, where $a^{(P)}_{m}$ are the Fourier coefficients of the weighted theta series 
\begin{eqnarray*}
\Theta _{\Lambda , P}(z)=\sum_{m\geq 0}a^{(P)}_{m}q^{m}. 
\end{eqnarray*}
\end{lem}

The theta series of $\Lambda $ weighted by $P$ is a modular form for some subgroup of $SL_{2}(\R)$. We recall the definition of the modular forms. 

\begin{df}
Let $\Gamma \subset SL_{2}(\R)$ be a Fuchsian group of the first kind and let $\chi$ be a character of $\Gamma$. A holomorphic function $f:\mathbb{H}\rightarrow \C$ is called a modular form of weight $k$ for $\Gamma$ with respect to $\chi$, if the following conditions are satisfied{\rm :} 
\begin{eqnarray*}
&{\rm (i)}& \quad f\Big(\frac{az+b}{cz+d}\Big)
=\Big(\frac{cz+d}{\chi (\sigma)}\Big)^{k}f(z)\ {\rm for}\ 
{\rm all}\ \sigma =
\begin{pmatrix}
a & b\\
c & d
\end{pmatrix}
 \in \Gamma .\\
&{\rm (ii)}& \quad f(z)\ {\rm is\ holomorphic\ at\ every\ cusp\ of}\ \Gamma. 
\end{eqnarray*}
\end{df}
If $f(z)$ has period $N$, then $f(z)$ has a Fourier expansion at infinity, \cite{Kob}:
\begin{eqnarray*}
f(z)=\sum_{m=0}^{\infty}a_{m} q_{N}^{m},\ q_{N}=e^{2 \pi i z /N}. 
\end{eqnarray*}
We remark that for $m<0$, $a_{m}=0$, by the condition (ii). A modular form with constant term $a_{0}=0$, is called a cusp form. 
We denote by $M_{k}(\Gamma , \chi)$ (resp. $S_{k}(\Gamma , \chi)$) the space of modular forms (resp. cusp forms) with respect to $\Gamma$ with the character $\chi$. 
When $f$ is the normalized eigenform of Hecke operators, p.163, \cite{Kob}, the Fourier coefficients satisfy the following relations:
\begin{lem}[cf.~{\cite[Proposition 32, 37, 40, Exercise 2, p.164]{Kob}}]\label{lem:lemrama}
Let $f(z)=\sum_{m\geq1}a(m)q^{m} \in S_{k}(\Gamma , \chi)$. 
If $f(z)$ is the normalized eigenform of Hecke operators, then the Fourier coefficients of $f(z)$ satisfy the following relations{\rm :}
\begin{eqnarray}
a(mn)&=&a(m)a(n)\ if\ (m, n)=1 \label{eqn:mul}\\ 
a(p^{\alpha +1})&=&a(p)a(p^{\alpha })-\chi(p) p^{k-1}a(p^{\alpha -1})\ 
if\ p\ is\ a\ prime. \label{eqn:rec}
\end{eqnarray}
\end{lem}
We set $f(z)=\sum_{m\geq1}a(m)q^{m} \in S_{k}(\Gamma , \chi)$. 
When $\dim S_{k}(\Gamma , \chi)=1$ and $a(1)=1$, then $f(z)$ is the normalized eigenform of Hecke operators, \cite{Kob}. So, the coefficients of $f(z)$ have the relations as mentioned in Lemma \ref{lem:lemrama}. 
It is known that 
\begin{eqnarray}
|a(p)|<2p^{(k-1)/2} \label{eqn:Deligne}
\end{eqnarray}
for all primes $p$. 
Note that this is the Ramanujan conjecture and its generalization, called the Ramanujan-Petersson conjecture for cusp forms which are eigenforms of the 
Hecke operators. These conjectures were proved by Deligne as a consequence of his proof of the Weil conjectures, \cite[page 164]{Kob}, \cite{Katz}. 
Moreover, for a prime $p$ with $\chi (p)=1$ the following equation holds, \cite{Lehmer}. 
\begin{eqnarray}
a(p^{\alpha })=p^{(k-1)\alpha /2}\frac{\sin (\alpha +1)\theta _{p}}{\sin \theta _{p}}, \label{eqn:Lehmer} 
\end{eqnarray}
where $2 \cos \theta _{p} = a(p)p^{-(k-1)/2}$.

It is well known that the theta series of $\Lambda \subset \R^{n}$ weighted by harmonic polynomial $P\in {\rm Harm}_{j}(\R^{n})$ is a modular form of weight $n/2+j$ for some subgroup $\Gamma \subset SL_{2}(\R)$. In particular, when $\deg (P)\geq 1$, the theta series of $\Lambda$ weighted by $P$ is a cusp form. 

For example, we consider the even unimodular lattice $\Lambda$. Then the theta series of $\Lambda $ weighted by harmonic polynomial $P$, $\Theta_{\Lambda, P}(z)$, is a modular form with respect to $SL_{2}(\Z)$. 

\begin{ex}
Let $\Lambda $ be the $E_{8}$-lattice. This is an even unimodular lattice of $\R^{8}$, generated by the $E_{8}$ root system. The theta series is as follows{\rm :}
\begin{eqnarray*}
\Theta_{\Lambda}(z)=E_{4}(z)&=&1+240\sum_{m=1}^{\infty}\sigma _{3}(m) q^{2m} \\
&=&1 + 240 q^2 + 2160 q^4 + 6720 q^6 + 17520 q^8 +\cdots, 
\end{eqnarray*}
where $\sigma _{3}(m)$ is a divisor function $\sigma _{3}(m)=\sum_{0<d|m}d^3$. 

For $j=2, 4$ and $6$, the theta series of $\Lambda $ weighted by $P\in {\rm Harm}_{j}(\R^{8})$ is a weight $6, 8$ and $10$ cusp form with respect to $SL_{2}(\Z)$. However, it is well known that for $k=6, 8$ and $10$, $\dim S_{k}(SL_{2}(\Z))=0$, that is, $\Theta_{\Lambda , P}(z)=0$. Then by Lemma \ref{lem:lempache}, 
all the nonempty shells of $E_{8}$-lattice are spherical $6$-design. 
\end{ex}

For $j=8$, the theta series of $\Lambda $ weighted by $P$ is a weight $12$ cusp form with respect to $SL_{2}(\Z)$. Such a cusp form is uniquely determined up to constant, i.e., it is Ramanujan's delta function:
\begin{eqnarray*}
\Delta _{24}(z)=q^2\prod_{m\geq 1}(1-q^{2m})^{24}=\sum_{m\geq 1}\tau (m)q^{2m}. 
\end{eqnarray*}

The following proposition is due to Venkov, de la Harpe and Pache 
\cite{{HP},{HPV},{Pache},{Venkov1}}. 
\begin{prop}[cf.~\cite{Pache}]\label{prop:proppache}
Let the notation be the same as above. 
Then the following are equivalent{\rm :}
\begin{itemize}
\item 
[\rm{(i)}]

$\tau (m)=0$. 

\item 
[\rm{(ii)}]
$(\Lambda )_{2m}$ is an $8$-design.
\end{itemize}
\end{prop}
It is a famous conjecture of Lehmer that $\tau (m) \neq 0$. So, Proposition \ref{prop:proppache} gives a reformulation of Lehmer's conjecture. 
Lehmer proved in \cite{Lehmer} the following theorem. 

\begin{thm}[cf.~\cite{Lehmer}]\label{thm:Leh}
Let $m_{0}$ be the least value of $m$ for which $\tau (m)=0$. 
Then $m_{0}$ is a prime if it is finite. 
\end{thm}

There are many attempts to study Lehmer's conjecture 
(\cite{{Lehmer},{Serre1}}), 
but it is difficult to prove and it is still open. 

Recently, however, we showed the ``Toy models for D. H. Lehmer's conjecture" 
\cite{Toy-BM}. 
We take the two cases $\Z^{2}$-lattice and $A_{2}$-lattice. Then, we consider the analogue of Lehmer's conjecture corresponding to the theta series weighted by some harmonic polynomial $P$. Namely, we show that the $m$-th coefficient of the weighted theta series of $\Z^{2}$-lattice does not vanish when the shell of norm $m$ of those lattices is not an empty set. Or equivalently, we show the following result.
\begin{thm}[cf.~\cite{Toy-BM}]\label{thm:toy}
The nonempty shells in $\mathbb{Z}^{2}$-lattice $($resp. $A_{2}$-lattice$)$ are not spherical $4$-designs $($resp. $6$-designs$)$. 
\end{thm}

This paper is sequel to the previous paper \cite{Toy-BM}. 
In this paper, 
we take some lattices related to the imaginary quadratic fields. 
Let $K=\Q(\sqrt{-d})$ be an imaginary quadratic field, 
and let $\mathcal{O}_{K}$ be its ring of algebraic integers. 
Let $\Cl_{K}$ be the ideal classes. In this paper, 
we only consider the cases $\vert \Cl_{K}\vert =1$ 
and $\vert \Cl_{K}\vert =2$ except for Section $6$. 
So, when we consider the cases $\vert \Cl_{K}\vert =1$ 
and $\vert \Cl_{K}\vert =2$, we denote by 
$\oo$ (resp. $\mathfrak{a}$) the principal 
(resp. nonprincipal) ideal class.

We denote by $d_{K}$ the discriminant of $K$:
\begin{eqnarray*}
d_{K}=\left\{
\begin{array}{lll}
-4d\ &{\rm if }\ -d\equiv 2,\ 3&\pmod {4}, \\
-d\ &{\rm if }\ -d\equiv 1 &\pmod {4}. 
\end{array}
\right.
\end{eqnarray*}
\begin{thm}[cf.~{\cite[page~87]{Zagier}}]\label{thm:lattice}
Let $d$ be a positive square-free integer, and let $K=\Q(\sqrt{-d})$. 
Then 
\begin{eqnarray*}
\mathcal{O}_{K}=
\left\{
\begin{array}{lll}
\ZZ+\ZZ\,\sqrt{-d}\quad &if\ -d\equiv 2,\ 3 &\pmod {4}, \\
\ZZ+\ZZ\,\displaystyle\frac{-1+\sqrt{-d}}{2}\quad &if\ -d\equiv 1 &\pmod {4}. 
\end{array}
\right.
\end{eqnarray*}
\end{thm}
Therefore, we consider $\OO_{K}$ to be the lattice in $\R^{2}$ 
with the basis 
\begin{eqnarray*}
\left\{
\begin{array}{lll}
\displaystyle(1, 0), (1,\sqrt{-d} )\quad &\mbox{if}\ -d\equiv 2,\ 3 &\pmod {4}, \\
\displaystyle(1, 0), \Big(-\frac{1}{2},\frac{\sqrt{-d}}{2}\Big)\quad &
\mbox{if}\ -d\equiv 1 &\pmod {4}, 
\end{array}
\right.
\end{eqnarray*}
denoted by $L_{\oo}$.

Generally, it is well-known that there exists one-to-one correspondence 
between the set of reduced quadratic forms $f(x, y)$ 
with a fundamental discriminant $d_K<0$ 
and the set of fractional ideal classes of the unique quadratic field 
$\QQ(\sqrt{-d})$ \cite[page~94]{Zagier}. Namely, 
For a fractional ideal $A=\ZZ\alpha +\ZZ\beta$, 
we obtain the quadratic form $ax^2+bxy+cy^2$, where 
$a=\alpha \alpha^{\prime}/N(A)$, 
$b=(\alpha\beta^{\prime}+ \alpha^{\prime}\beta)/N(A)$ and 
$c=\beta \beta^{\prime}/N(A)$. 
Conversely, for a quadratic form $ax^2+bxy+cy^2$, 
we obtain the fractional ideal $\ZZ +\ZZ(b+\sqrt{d_K})/2a$. 
We remark that $N(A)$ is a norm of $A$ and 
$\alpha^{\prime}$ is a complex conjugate of $\alpha$. 

Here, we define the automorphism group of $f(x, y)$ as follows: 
\begin{eqnarray*}
U_{f}=\left\{
\begin{pmatrix}
\alpha &\beta \\
\gamma &\delta
\end{pmatrix}
\in SL_{2}(\ZZ) \Biggm\vert f(\alpha x+\beta y, \gamma x+\delta y)=f(x, y)
\right\}. 
\end{eqnarray*}
Then, for $n\geq 1$, the number of the nonequivalent solutions of $f(x, y)=n$ 
under the action of $U_{f}$ 
is equal to the number of the integral ideals of norm $n$ \cite{Zagier}. 
\begin{thm}[cf.~{\cite[page~63]{Zagier}}]\label{thm:Zagier}
Let $f(x, y)$ be the reduced quadratic form 
with a fundamental discriminant $D<0$ and 
$U_{f}$ be the automorphism group of $f(x, y)$. 
Then 
\begin{eqnarray*}
\sharp U_{f}=
\left\{
\begin{array}{ll}
6 &{\text if}\ D=-3, \\
4 &{\text if}\ D=-4, \\
2 &{\text if}\ D<-4. 
\end{array}
\right.
\end{eqnarray*}
\end{thm}
These classical results are due to Gauss, Dirichlet, etc. 
Let $\mathfrak{a}$ be an ideal class and 
$f_{\mathfrak{a}}(x, y)$ be the reduced quadratic form 
corresponding to $\mathfrak{a}$. 
Moreover, let $L_{\mathfrak{a}}$ be the lattice corresponding to $f(x, y)$. 
We denote by $N(A)$ the norm of an ideal $A$. 
Then, using Theorem \ref{thm:Zagier}, we have 
\begin{equation}\label{eqn:multi}
\begin{array}{c}
\displaystyle\sum_{x\in L_{\mathfrak{a}}}q^{(x, x)}\hspace{300pt}  \\
=1+\sharp U_{f}\sum_{n=1}^{\infty}
\sharp\{A \mid A\mbox{ is an integral ideal of }\mathfrak{a},\, 
N(A)=n\}\,q^{m}. 
\end{array}
\end{equation}
When $\vert \Cl_{K}\vert =2$, we give the generators of $L_{\mathfrak{a}}$
 in Appendix. 
Here, we remark that when $K=\Q(\sqrt{-1})$ (resp. $K=\Q(\sqrt{-3})$), 
$L_{\oo}$ is $\Z^{2}$-lattice (resp. $A_{2}$-lattice). 
We studied the spherical designs of shells of 
those lattices in the previous paper \cite{Toy-BM}. 

In this paper, we take the imaginary quadratic fields $\QQ(\sqrt{-d})$, 
with $d\neq 1$ and $d\neq 3$. 
Then, we consider the analogue of Lehmer's conjecture 
corresponding to its theta series weighted by some harmonic polynomial $P$.
Here, we consider the following problem that 
whether the nonempty shells of $L_{\oo}$ and $L_{\mathfrak{a}}$ 
are spherical $2$-designs (hence $3$-designs) or not. 

In Section $4$, we study the case that the class number is $1$. 
We show that the $m$-th coefficient of 
the weighted theta series of $L_{\oo}$-lattice does not vanish 
when the shell of norm $m$ of those lattices is not an empty set. 
Or equivalently, we show the following result: 
\begin{thm}\label{thm:classnumber=1}
Let $K=\Q(\sqrt{-d})$ be an imaginary quadratic field 
whose class number is $1$ and $d\neq 1$, $3$ 
i.e., $d$ is in the following set$:$ 
$\{2$, $7$, $11$, $19$, $43$, $67$, $163\}$. 
Then, the nonempty shells in $L_{\oo}$ are not spherical $2$-designs. 
\end{thm}
Similarly, in Section $5$, we study the case that the class number is $2$ 
and show the following result: 
\begin{thm}\label{thm:classnumber=2}
Let $K=\Q(\sqrt{-d})$ be an imaginary quadratic field 
whose class number is $2$ i.e., $d$ is in the following set$:$ 
$\{5$, $6$, $10$, $13$, $15$, $22$, $35$, $37$, $51$, $58$, 
$91$, $115$, $123$, $187$, $235$, $267$, $403$, $427\}$. 
Then, the nonempty shells in $L_{\oo}$ and $L_{\mathfrak{a}}$ are not spherical $2$-designs. 
\end{thm}
In Section $6$, we consider the case that the class number is $3$ 
and study the property of Heche characters. 
In Section $7$, we give some concluding remarks and 
state a conjecture for the future study. 
\section{Preliminaries}
In this section, we review the theory of imaginary quadratic fields. 
\begin{thm}[cf.~{\cite[page 104,~Proposition 5.16]{Cox}}]\label{thm:facprime}
We can classify the prime ideals of a quadratic field as follows{\rm :}
\begin{enumerate}
\item 
If $p$ is an odd prime and $(d_{K}/p)=1$ 
$($resp.\ $d_{K}\equiv 1 \pmod{8}$$)$ then 
$(p)=P \overline{P}\ (resp.\ (2)=P \overline{P})$, 
where $P$ and $\overline{P}$ are prime ideals 
with $P\neq \overline{P}$, 
$N(P)=N(\overline{P})=p$ $($resp.\ $N(P)=2$$)$. 
\item 
If $p$ is an odd prime and $(d_{K}/p)=-1$ 
$($resp.\ $d_{K}\equiv 5 \pmod{8}$$)$ then 
$(p)=P\ (resp.\ (2)=P)$, 
where $P$ is a prime ideal with $N(P)=p^{2}$ $($resp.\ $N(P)=4$$)$. 
\item 
If $p\ \vert\ d_{k}$ then 
$(p)=P^{2}$, 
where $P$ is a prime ideal with $N(P)=p$. 
\end{enumerate}
\end{thm}

\begin{lem}\label{lem:Takagi}
Let $I$ be an integral ideal of $K$. 
For $n\in \NN$, 
if $N(I)=n$ and $I$ is a principal ideal, namely, $I\in \oo$ 
then there exist $a$, $b\in \Z$ 
such that for $-d\equiv 2,\ 3 \pmod{4}$ 
\begin{eqnarray*}
n=a^2+db^2, 
\end{eqnarray*}
for $-d\equiv 1 \pmod{4}$ 
\begin{eqnarray*}
n=a^2+db^2\quad or\quad n=\frac{a^2+db^2}{4}. 
\end{eqnarray*}

If $\vert \Cl_{K} \vert =2$, $N(I)=n$ and 
$I$ is a nonprincipal ideal, namely, $I\in \mathfrak{a}$ 
and assume that $m$ is one of the norm of nonprincipal ideals 
then there exist $a$, $b\in \Z$ such that for $-d\equiv 2,\ 3 \pmod{4}$ 
\begin{eqnarray*}
mn=a^2+db^2, 
\end{eqnarray*}
for $-d\equiv 1 \pmod{4}$ 
\begin{eqnarray*}
mn=a^2+db^2\quad or\quad mn=\frac{a^2+db^2}{4}. 
\end{eqnarray*}
\end{lem}

\begin{proof}
We assume that $\vert \Cl_{K} \vert =1$. 
For $-d\equiv 2,\ 3 \pmod{4}$, we can write $I=(a+b\sqrt{-d})$, 
then $N(I)=a^2+db^2$. 
For $-d\equiv 1 \pmod{4}$, we can write $I=(a+b\sqrt{-d})$ or 
$I=((a+b\sqrt{-d})/2)$, 
then $N(I)=a^2+db^2$ or $N(I)=(a^2+db^2)/4$. 

Here, we assume that $\vert \Cl_{K} \vert =2$. 
Let $J$ be the nonprincipal ideal of $K$ whose norm is $m$. 
If $I$ is a nonprincipal ideal then, $JI$ is a principal ideal of $K$. 
Therefore, for $-d\equiv 2,\ 3 \pmod{4}$, we can write $JI=(a+b\sqrt{-d})$, 
then $N(JI)=a^2+db^2$. Hence, $mn=a^2+db^2$. 
for $-d\equiv 1 \pmod{4}$, we can write $JI=(a+b\sqrt{-d})$ or 
$JI=((a+b\sqrt{-d})/2)$, 
then $N(JI)=a^2+db^2$ or $N(JI)=(a^2+db^2)/4$. 
Hence, $mn=a^2+db^2$ or $mn=(a^2+db^2)/4$. 
\end{proof}

\begin{prop}\label{prop:NUM}
Let $F(m)$ be the number of the integral ideals of norm $m$ of $K$. 
Let $p$ be a prime number. 
Then, if $p\neq 2$ 
\begin{eqnarray*}
F(p^{e})=\left\{
\begin{array}{lll}
e+1 &{ if}\ \left(d_{K}/p\right)=1, \\
(1+(-1)^e)/2 &{ if}\ \left(d_{K}/p\right)=-1, \\
1 &{ if}\ p\ \vert\ d_{K}, 
\end{array} 
\right.
\end{eqnarray*}
if $p= 2$ 
\begin{eqnarray*}
F(2^{e})=\left\{
\begin{array}{lll}
e+1 &{ if}\ d_{K}\equiv 1 \pmod{8}, \\
(1+(-1)^e)/2 &{ if}\ d_{K}\equiv 5 \pmod{8}, \\
1 &{ if}\ 2\ \vert\ d_{K}. 
\end{array} 
\right.
\end{eqnarray*}
\end{prop}
\begin{proof}
When $\left(d_{K}/p\right)=1$ i.e., $(p)=P \overline{P}$ and 
$P \neq \overline{P}$, 
since $P$ and $\overline{P}$ are the only integral ideals 
of norm $p$, 
we have $F(p)=2$. 
Moreover, the integral ideals of norm $p^{e}$ are as follows: 
$P^{e}$, $P^{e-1} \overline{P}$, \ldots , 
$(\overline{P})^{e}$. So, we have $F(p^{e})=e+1$. 
The other cases can be proved similarly. 
\end{proof}


\section{Hecke characters and Theta series}
In this section, we introduce the Hecke character and 
discuss the relationships between the Hecke character and the 
weighted theta series of the lattices $L_{\oo}$ and $L_\mathfrak{a}$. 
Then, we show that for $\vert \Cl_{K}\vert =1$ and $P_1=(x^2-y^2)/2$, 
the weighted theta series $\Theta_{L_{\oo}, P_1}$ 
is a normalized Hecke eigenform. 
For $\vert \Cl_{K}\vert =2$ and $P_2=x^2-y^2$, 
a certain sum of the two weighted theta series $c_1\Theta_{L_{\oo}, P_2}+
c_2\Theta_{L_{\mathfrak{a}}, P_2}$ is a normalized Hecke eigenform. 
Later, we give the explicit values of $c_1$ and $c_2$. 

A Hecke character $\phi$ of weight $k\geq 2$ 
with modulus $\Lambda$ is defined in the following way. 
Let $\Lambda $ be a nontrivial ideal in $\OO_{K}$ and 
let $I(\Lambda)$ denote the group of fractional ideals prime to 
$\Lambda$. A Hecke character $\phi$ with modulus $\Lambda$ 
is a homomorphism 
\begin{eqnarray*}
\phi : I(\Lambda)\rightarrow \C^{\times} 
\end{eqnarray*}
such that for each $\alpha \in K^{\times}$ with 
$\alpha \equiv 1 \pmod{\Lambda}$ we have 
\begin{eqnarray}
\phi(\alpha \OO_{K}) =\alpha^{k-1}. \label{eqn:principal}
\end{eqnarray}
Let $\omega_{\phi}$ be the Dirichlet character with the property 
that 
\begin{eqnarray*}
\omega_{\phi}(n):=\phi((n))/n^{k-1} 
\end{eqnarray*}
for every integer $n$ coprime to $\Lambda$. 

\begin{thm}[cf.~{\cite[page~9]{Web}},~~{\cite[page~183]{Miyake}}]\label{thm:ono}
Let the notation be the same as above, 
and define $\Psi_{K, \Lambda}(z)$ by 
\begin{eqnarray}
\Psi_{K, \Lambda}(z) :=\sum _{ A}\phi(A)q^{N(A)}
=\sum_{n=1}^{ \infty}a(n)q^{n}, \label{eqn:ono} 
\end{eqnarray}
where the sum is over the integral ideals $A$ that are prime 
to $\Lambda$ and $N(A)$ is the norm of the ideal $A$. 
Then $\Psi_{K, \Lambda}(z)$ is a cusp form in $S_{k}(\Gamma_{0}(d_K\cdot N(\Lambda)), 
\left( \frac{-d_K}{\bullet} \right)\omega_{\phi})$. 
\end{thm}
We remark that function (\ref{eqn:ono}) is a normalized Hecke eigenform 
\cite{{Ahlgren},{Serre}}. 
Moreover, if the class number of $K$ is $h$ then 
the character as given in (\ref{eqn:principal}) 
will have $h$ extensions to nonprincipal ideals. 
Namely, the function (\ref{eqn:ono}) has $h$ choices, 
so we denote by 
$\Psi_{K, \Lambda}^{(1)}(z), \ldots , \Psi_{K, \Lambda}^{(h)}(z)$ 
each functions (see \cite{Prime}). 
\begin{ex} $\\ $
\vspace{-20pt}
\begin{enumerate}
\item 
[(i)] $d=2$.\\ 
We calculate $\Psi_{K, \Lambda}(z)=\sum_{m\geq 1}a(m)q^{m}$, 
where $\Lambda=(1)$ and the weight of 
the Hecke character is $3$. We remark that $\vert \Cl_{K}\vert =1$. 
\begin{table}[thb]
\caption{Integral ideals of small norm of $d=2$ and $d=5$}
\label{Tab:d=2,d=5}
\begin{center}
{\footnotesize
\begin{tabular}{c||c} 
\noalign{\hrule height0.8pt}
\hline
$N(A)$ & $A$: ideal  \\ \hline
$1$& $(1)$ \\ \hline
$2$& $(\sqrt{-2})$ \\ \hline
$3$& $(-1+\sqrt{-2})$ \\ 
   & $(-1-\sqrt{-2})$ \\ \hline
$4$& $(2)$ \\ \hline
\noalign{\hrule height0.8pt}
\end{tabular}
\hspace{10pt}
\begin{tabular}{c||c} 
\noalign{\hrule height0.8pt}
\hline
$N(A)$ & $A$: ideal  \\ \hline
$1$& \MC{1}{c}{$(1)$} \\ \hline
$2$& \MC{1}{c}{$(2, 1+\sqrt{-5})$} \\ \hline
$3$& $(3, 1+\sqrt{-5})$ \\ 
   & $(3, 1-\sqrt{-5})$ \\ \hline
$4$& \MC{1}{c}{$(2)$} \\ \hline
$5$& \MC{1}{c}{$(\sqrt{-5})$} \\ \hline
$6$& $(1-\sqrt{-5})$ \\ 
   & $(-1-\sqrt{-5})$ \\ \hline
\noalign{\hrule height0.8pt}
\end{tabular}
}
\end{center}
\end{table}
By the definitions (\ref{eqn:principal}) and (\ref{eqn:ono}), 
we have $a(1)=1^2=1$, $a(2)=\sqrt{-2}^2=-2$, 
$a(3)=(-1+\sqrt{-2})^2+(-1-\sqrt{-2})^2=2$, $a(4)=2^2$, \ldots. 
Thus, we obtain 
\begin{eqnarray*}
\Psi_{K, \Lambda}^{(1)}(z) =q - 2 q^2 - 2 q^3 + 4 q^4 + 4 q^6 - 8 q^8 - 5 q^9  +\cdots .
\end{eqnarray*}

\item 
[{\rm (ii)}] $d=5$.\\ 
We calculate $\Psi_{K, \Lambda}(z)=\sum_{m\geq 1}a(m)q^{m}$, 
where $\Lambda=(1)$ and 
the weight of the Hecke character is $3$. 
We remark that $\vert \Cl_{K}\vert =2$. 
When $A$ of norm $m$ is a nonprincipal ideal, 
$A^{2}$ is a principal ideal, so, $\phi (A^{2})$ is computable 
by the definition (\ref{eqn:principal}). 
For example, $\phi((2, 1+\sqrt{-5}))^{2}=\phi((2))=4$, 
so, we can assume that $\phi((2, 1+\sqrt{-5}))=2$, i.e., $a(2)=2$. Then, 
since $(2, 1+\sqrt{-5})(3, 1+\sqrt{-5})=(1-\sqrt{-5})$ and 
$(2, 1+\sqrt{-5})(3, 1-\sqrt{-5})=(-1-\sqrt{-5})$, we have 
$a(3)=((1+\sqrt{-5})^2+(1-\sqrt{-5})^2)/2=-4$, $a(4)=2^2$, \ldots. 
Thus, we obtain 
\begin{eqnarray*}
\Psi_{K, \Lambda}^{(1)}(z) =q+2 q^2-4 q^3+4 q^4-5 q^5-8 q^6+4 q^7+8 q^8+7 q^9+\cdots .
\end{eqnarray*}
On the other hand, we assume that $\phi((2, 1+\sqrt{-5}))=-2$, 
i.e., $a(2)=-2$. Then, we have 
\begin{eqnarray*}
\Psi_{K, \Lambda}^{(2)}(z) =q-2 q^2+4 q^3+4 q^4-5 q^5-8 q^6-4 q^7-8 q^8+7 q^9+\cdots .
\end{eqnarray*}
\end{enumerate}
\end{ex}

Here, we discuss the relationships between the Hecke character and the 
weighted theta series of the lattices $L_{\oo}$ and $L_\mathfrak{a}$. 
First, we quote the following theorem: 
\begin{thm}[cf.~{\cite[page~192]{Miyake}}]\label{thm:Miyake}
Let $L$ be an integral lattice with the Gram matrix $A$ and 
$N$ be the natural number such that the elements of $NA^{-1}$ 
are rational integers. 
Let the character $\chi (d)$ be 
\begin{eqnarray*}
\chi (d)=\Big(\frac{(-1)^{(r/2)}\det L}{d}\Big). 
\end{eqnarray*}
Then, for $P\in$ {\rm Harm}$_{2}(\RR^{2})$, 
\begin{enumerate}
\item 
[{\rm (1)}]
$\Theta_{L,P}\in M_{3}(\Gamma_0 (4N),\chi)$. 

\item 
[{\rm (2)}]
If all the diagonal elements of $A$ are even, 
then $\Theta_{L,P}\in M_{3}(\Gamma_0 (2N),\chi)$. 

\item 
[{\rm (3)}] 
If all the diagonal elements of $A$ and $NA^{-1}$ 
are even, then 
$\Theta_{L,P}\in M_{3}(\Gamma_0 (N),\chi)$. 
\end{enumerate}
\end{thm}

Then, we obtain the following lemmas: 
\begin{lem}\label{lem:ono}
Let $K$ be an imaginary quadratic field whose class number is $1$ and 
$L_{\oo}$ be the lattice corresponding to the principal ideal class $\oo$. 
Let $\phi$ be the Hecke character of weight $3$ 
with modulus $\Lambda$. 
Assume that $\Lambda =(1)$ and $P_1=(x^2-y^2)/2\in$ {\rm Harm}$_{2}(\R^{2})$. 
Then, $\Psi_{K, \Lambda}(q)=\Theta_{L_{\oo},P_1}(q)$. 
\end{lem}

\begin{lem}\label{lem:ono2}
Let $K$ be an imaginary quadratic field whose class number is $2$ and 
$L_{\oo}$ $($resp. $L_{\mathfrak{a}}$$)$ 
be the lattice corresponding to the principal ideal class $\oo$ 
$($resp. nonprincipal ideal class $\mathfrak{a}$$)$. 
Let $\phi$ be the Hecke character of weight $3$ 
with modulus $\Lambda$. 
Assume that $\Lambda =(1)$ and $P_2=x^2-y^2\in$ {\rm Harm}$_{2}(\R^{2})$. 
Then, 
$\Psi_{K, \Lambda}(q)=c_{1}\Theta_{L_{\oo},P_2}(q)
+c_{2}\Theta_{L_{\mathfrak{a}},P_2}(q)$, 
where $c_{1}$ and $c_{2}$ are given as in table \ref{Tab:c_1,c_2}$.$\\ 
\begin{table}[thb]
\caption{Coefficients, $c_1$ and $c_2$}
\label{Tab:c_1,c_2}
\begin{center}
{\footnotesize
\begin{tabular}{c||lllllllll}
\noalign{\hrule height0.8pt}
\hline
$-d$               & \MC{1}{c}{$-5$} & \MC{1}{c}{$-6$} & \MC{1}{c}{$-10$} 
                   & \MC{1}{c}{$-13$} & \MC{1}{c}{$-15$} & \MC{1}{c}{$-22$}
                   & \MC{1}{c}{$-35$} & \MC{1}{c}{$-37$} & \MC{1}{c}{$-51$}\\
\hline
$c_{1}$ &  \MC{1}{c}{$1/2$} & \MC{1}{c}{$1/2$} &\MC{1}{c}{$1/2$}  
        &  \MC{1}{c}{$1/2$} & \MC{1}{c}{$1/2$} &\MC{1}{c}{$1/2$}  
        &  \MC{1}{c}{$1/2$} & \MC{1}{c}{$1/2$} &\MC{1}{c}{$1/2$}  \\
$c_{2}$ &  \MC{1}{c}{$1/2$} & \MC{1}{c}{$1/2$} &\MC{1}{c}{$1/2$}  
        &  \MC{1}{c}{$1/2$} & \MC{1}{c}{$2$} &\MC{1}{c}{$1/2$}  
        &  \MC{1}{c}{$3$} & \MC{1}{c}{$1/2$} &\MC{1}{c}{$1/2$}  \\
\hline
\hline
$-d$               & \MC{1}{c}{$-58$} & \MC{1}{c}{$-91$} & \MC{1}{c}{$-115$} 
                   & \MC{1}{c}{$-123$} & \MC{1}{c}{$-187$} & \MC{1}{c}{$-235$}
                   & \MC{1}{c}{$-267$} & \MC{1}{c}{$-403$} & \MC{1}{c}{$-427$}\\
\hline
$c_{1}$ &  \MC{1}{c}{$1/2$} & \MC{1}{c}{$1/2$} &\MC{1}{c}{$1/2$}  
        &  \MC{1}{c}{$1/2$} & \MC{1}{c}{$1/2$} &\MC{1}{c}{$1/2$}  
        &  \MC{1}{c}{$1/2$} & \MC{1}{c}{$1/2$} &\MC{1}{c}{$1/2$}  \\
$c_{2}$ &  \MC{1}{c}{$1/2$} & \MC{1}{c}{$5/3$} &\MC{1}{c}{$1/2$}  
        &  \MC{1}{c}{$1/2$} & \MC{1}{c}{$7/3$} &\MC{1}{c}{$1/2$}  
        &  \MC{1}{c}{$1/2$} & \MC{1}{c}{$11/9$} &\MC{1}{c}{$1/2$}  \\
\hline
\noalign{\hrule height0.8pt}
   \end{tabular}
}
\end{center}
\end{table}
\hspace{-5pt}{\it Proof of Lemmas 3.1 and 3.2.}
First, we assume that the lattices are integral lattices, 
if not we multiple the Gram matrix of $L$ by $2$. 

{\rm Because of the Theorems \ref{thm:ono} and \ref{thm:Miyake}, 
$\Psi_{K, \Lambda}(q)$, $\Theta_{L_{\oo},P}(q)$ 
and $\Theta_{L_{\mathfrak{a}},P}(q)$ with $P=P_1$, $P_2$ are 
modular forms of the same group $\Gamma$. 
Therefore, we calculate the basis of the space of modular forms 
of group $\Gamma$ 
and check $\Psi_{K, \Lambda}(q)=\Theta_{L_{\oo},P_1}(q)$ and 
$\Psi_{K, \Lambda}(q)=c_{1}\Theta_{L_{\oo},P_2}(q)
+c_{2}\Theta_{L_{\mathfrak{a}},P_2}(q)$ explicitly 
(using ``Sage", Mathematics Software \cite{Sage}).} \e
\end{lem}


\begin{cor}\label{cor:HeckeTheta}
Let the notation be the same as above. 
If $\vert \Cl_{K}\vert=1$ then $\Theta_{L_{1},P_1}(q)$ is 
a normalized Hecke eigenform. 
If $\vert \Cl_{K}\vert=2$ then 
$c_{1}\Theta_{L_{1},P_2}(q)+c_{2}\Theta_{L_{2},P_2}(q)$ is 
a normalized Hecke eigenform. 
\end{cor}
\begin{proof}
The function (\ref{eqn:ono}) is a normalized Hecke eigenform 
\cite{{Ahlgren},{Serre}}. 
\end{proof}

Finally, we give the following proposition, 
which is an analogue of Theorem \ref{thm:Leh} 
and the crucial part of the proof of Theorems \ref{thm:classnumber=1} 
and \ref{thm:classnumber=2}. 
\begin{prop}\label{prop:Lehmer}
Assume that $\sum_{m\geq 1}a(m)q^{m}$ is a normalized Hecke eigenform of 
$S_{3}(\Gamma, \chi)$ and the coefficients $a(m)$ are rational integers. 
Moreover Let $p$ be the prime such that $\chi (p)=1$. 
Let $\alpha_0$ be the least value of $\alpha$ for which $a{\rm (}p^{\alpha}{\rm )}=0$. If $a(p)\neq \pm p$ then $\alpha_0 =1$ if it is finite. 
\end{prop}

\begin{proof}
Assume the contrary, that is, $\alpha_0 > 1$, 
so that $a(p)\neq 0$. By the equation (\ref{eqn:Lehmer}), 
\begin{eqnarray*}
a(p^{\alpha_0})=0=p^{\alpha_0}\frac{\sin (\alpha_{0}+1)\theta _{p}}{\sin \theta _{p}}. 
\end{eqnarray*}
This shows that $\theta _{p}$ is a real number of the form
$\theta _{p}=\pi k / (1+\alpha_0)$, where $k$ is an integer. Now the number 
\begin{eqnarray}
z=2 \cos\theta _{p}=a(p) p^{-1}, \label{lem:1}
\end{eqnarray}
being twice the cosine of a rational multiple of $2 \pi$, is an algebraic integer. On the other hand $z$ is a root of the equation 
\begin{eqnarray}
pz-a(p)=0.  \label{lem:2}
\end{eqnarray}
Hence $z$ is a rational integer. By (\ref{eqn:Deligne}) and (\ref{lem:1}), we have $\vert z\vert \leq 1$. Therefore $z=\pm 1$ and the equation (\ref{lem:2}) becomes $a(p)= \pm p$. By assumption, this is a contradiction. 
\end{proof}

\section{The case of $\vert \Cl_{K}\vert =1$}

Let $K:=\Q(\sqrt{-d})$ be an imaginary quadratic field. 
If the class number of $K$ 
is $1$ then $d$ is in the following set $\{1$, 
$2$, $3$, $7$, $11$, $19$, $43$, $67$, $163\}$. 
In particular, 
we only consider the cases where $d$ is in the set: 
$\{$$2$, $7$, $11$, $19$, $43$, $67$, $163\}$ 
since the cases $d=1$ and $d=3$ are considered in \cite{Toy-BM}. 

In this section, we assume that $a(m)$ and $b(m)$ 
are the coefficients of the following functions: 
\begin{eqnarray*}
\Theta_{L_{\oo}}(q)=\sum_{m\geq 0}a(m)q^{m},\ 
\Theta_{L_{\oo},P_1}(q)=\sum_{m\geq 1}b(m)q^{m}, 
\end{eqnarray*}
where $P_1=(x^2-y^2)/2 \in {\rm Harm}_{2}(\R^{2})$. 

\begin{lem}\label{lem:num}
Let $d$ be one of the elements in $\{2$, $7$, $11$, $19$, $43$, $67$, $163$\}. 
We set $a^{\prime}(m)=a(m)/2$ for all $m$. Then, 
\begin{eqnarray*}
a^{\prime}(p^{e})=\left\{
\begin{array}{lll}
e+1 &{ if}\ \left(d_{K}/p\right)=1, \\
(1+(-1)^e)/2 &{ if}\ \left(d_{K}/p\right)=-1, \\
1 &{ if}\ p\ \vert\ d_{K}. 
\end{array} 
\right.
\end{eqnarray*}
\end{lem}
\begin{proof}
Because of the equation (\ref{eqn:multi}), 
$a^{\prime}(m)$ is the number of integral ideals of $K$ of norm $m$. 
Therefore, it can be proved by Proposition \ref{prop:NUM}. 
\end{proof}

\begin{lem}\label{lem:non0}
Let $p$ be a prime number such that $(d_{K}/p)=1$. 
Then, $b(p)\neq 0$. Moreover, if $p\neq d$ then $b(p)\neq \pm p$. 
\end{lem}
\begin{proof}
We remark that by Corollary \ref{cor:HeckeTheta}, 
$\Theta _{L_{\oo},P_1}(q)=\Psi_{K,\Lambda}(q)$. 
So, the numbers $b(m)$ are the coefficients of $\Psi_{K,\Lambda}(q)$. 

First, we assume that $d\neq 2$ i.e.,\ $-d\equiv 1 \pmod 4$ and 
$\OO_{K}=\ZZ+\ZZ(1+\sqrt{-d})/2$. 
If $N((a+b\sqrt{-d}))$ is equal to $p$ then by Lemma \ref{lem:Takagi} 
\begin{eqnarray*}
p=a^{2}+db^{2}. 
\end{eqnarray*}
Because of the definition of $\Psi_{K,\Lambda}(q)$, 
\begin{eqnarray*}
b(p)= (a+b\sqrt{-d})^2 + 
(a-b\sqrt{-d})^2=2(a^{2}-db^{2}). 
\end{eqnarray*}
If $b(p)= 0$ then $a^{2}=db^{2}$. This is a contradiction. 
Assume that $b(p)= \pm p$. Then, 
\begin{eqnarray*}
2(a^{2}-db^{2})=\pm (a^{2}+db^{2}), 
\end{eqnarray*}
that is, $a^{2}=3db^{2}$ or $3a^{2}=db^{2}$. 
This is a contradiction.

If $N(((a+b\sqrt{-d})/2))$ is equal to $p$ then by Lemma \ref{lem:Takagi} 
\begin{eqnarray*}
\frac{a^{2}+db^{2}}{4}=p. 
\end{eqnarray*}
Because of the definition of $\Psi_{K,\Lambda}(q)$, 
\begin{eqnarray*}
b(p)= \Big(\frac{a+b\sqrt{-d}}{2}\Big)^2 + 
\Big(\frac{a-b\sqrt{-d}}{2}\Big)^2=\frac{a^{2}-db^{2}}{2}. 
\end{eqnarray*}
If $b(p)= 0$ then $a^{2}=db^{2}$. This is a contradiction. 
Assume that $b(p)= \pm p$. Then, 
\begin{eqnarray*}
\frac{a^{2}-db^{2}}{2}=\pm \frac{a^{2}+db^{2}}{4}, 
\end{eqnarray*}
that is, $a^{2}=3db^{2}$ or $3a^{2}=db^{2}$. 
This is a contradiction. 

Next, we assume that $d= 2$ i.e., $-d\equiv 2 \pmod 4$ and 
$\OO_{K}=\ZZ+\ZZ\sqrt{-2}$. 
If $N((a+b\sqrt{-2}))$ is equal to $p$ then by Lemma \ref{lem:Takagi} 
\begin{eqnarray*}
p=a^{2}+2b^{2}. 
\end{eqnarray*}
Because of the definition of $\Psi_{K,\Lambda}(q)$, 
\begin{eqnarray*}
b(p)= (a+b\sqrt{-2})^2 + (a-b\sqrt{-2})^2=2(a^{2}-2b^{2}). 
\end{eqnarray*}
If $b(p)= 0$ then $a^{2}=2b^{2}$. This is a contradiction. 
Assume that $b(p)= \pm p$. Then, 
\begin{eqnarray*}
2(a^{2}-2b^{2})=\pm (a^{2}+2b^{2}), 
\end{eqnarray*}
that is, $a^{2}=6b^{2}$ or $3a^{2}=2b^{2}$. 
This is a contradiction. 
\end{proof}
\hspace{-17pt}{\bf\it Proof of Theorem \ref{thm:classnumber=1}. }
We will show that $b(m)\neq 0$ when $(L_{\oo})_{m}\neq \emptyset$. 

By Theorem \ref{thm:ono}, $\Theta_{L_{\oo},P_1}$ is a normalized Hecke 
eigenform. So, We assume that $m$ is a power of prime, 
if not we could apply the equation (\ref{eqn:mul}).
We will divide our considerations into the following three cases. 
\begin{enumerate}
\item 
[(i)]
Case $m=p^{\alpha}$ and $p\ \vert\ d_{K}$: \\
By $a(m)=2$ and the inequality (\ref{ine:Fisher}), 
the shells $(L_{\oo})_{m}$ are not spherical $2$-designs. 
Hence, $b(m)\neq 0$. 

\item 
[(ii)] Case $m=p^{\alpha}$ and $\left(d_{K}/p\right)=-1$: \\
By Lemma \ref{lem:num}, 
\begin{eqnarray*}
a(p^{n})= 
\left\{
\begin{array}{llll}
0  \quad &{\rm if\ }n\ {\rm is\ odd}, \\
2  \quad &{\rm if\ }n\ {\rm is\ even}. 
\end{array} 
\right.
\end{eqnarray*}
By $a(m)=2$ and the inequality (\ref{ine:Fisher}), when $n$ is even, 
the shells $(L_{\oo})_{m}$ are not spherical $2$-designs. 
Hence, $b(m)\neq 0$. 

\item 
[(iii)] Case $m=p^{\alpha}$ and $\left(d_{K}/p\right)=1$: \\
By Proposition \ref{prop:Lehmer} and Lemma \ref{lem:non0}, 
we have $b(m)\neq 0$. 
This completes the proof of Theorem \ref{thm:classnumber=1}. \e
\end{enumerate}
\section{The case of $\vert \Cl_{K}\vert =2$}

Let $K:=\Q(\sqrt{-d})$ be an imaginary quadratic field. 
In this section, we assume that the class number of $K$ is $2$. 
So, we consider that $d$ is in the following set: 
$\{$$5$, $6$, $10$, $13$, $15$, 
$22$, $35$, $37$, $51$, $58$, $91$, $115$, $123$, $187$, $235$, 
$267$, $403$, $427\}$. 
We denote by 
$\oo$ (resp. $\mathfrak{a}$) the principal 
(resp. nonprincipal) ideal class. 

In this section, we also assume that $a(m)$ and $b(m)$ 
are the coefficients of the following functions: 
\begin{eqnarray*}
&&\Theta_{L_{\oo}}(q)+\Theta_{L_{\mathfrak{a}}}(q)=\sum_{m\geq 0}a(m)q^{m},\\ 
&&c_{1}\Theta_{L_{\oo},P_2}(q)+c_{2}\Theta_{L_{\mathfrak{a}},P_2}(q)=
\sum_{m\geq 1}b(m)q^{m}, 
\end{eqnarray*}
where $c_1$ and $c_2$ are defined in Lemma \ref{lem:ono2}. 

\begin{lem}\label{lem:coprime}
Set $l_{1}:=\{N(O)\ \vert\ x\in L_{\oo}\}$ and 
$l_{2}:=\{N(A)\ \vert\ A\in \mathfrak{a}\}$. 
Then, $l_{1}\cap l_{2}=\emptyset$. 
Therefore, the set $L_{\oo} \cap L_\mathfrak{a}$ 
consists of the origin. 
\end{lem}
\begin{proof}
Let $p$ be the prime number such that $(d_{K}/p)=1$. 
Then there exist prime ideals $P$ and $P^{\prime}$ 
such that $(p)=PP^{\prime}$ and $N(P)=N(P^{\prime})=p$. 
Since the class number is $2$, we have 
$P$ and $P^{\prime}\in \oo$ or 
$P$ and $P^{\prime}\in \mathfrak{a}$. 
If $P$ and $P^{\prime}\in \oo$ we denote by $p_{i}$ such a prime. 
If $P$ and $P^{\prime}\in \mathfrak{a}$ 
we denote by $p^{\prime}_{i}$ such a prime. 

Let $p$ be the prime number such that $(d_{K}/p)=-1$. 
Then $(p)$ is the prime ideal and $N((p))=p^{2}$. 
We denote by $q_{i}$ such a prime. 

Let $p$ be the prime number such that $p\ \vert\ d_{K}$. 
Then there exists a prime ideal $P$ 
such that $(p)=P^{2}$ and $N(P)=p$. 
Since the class number is $2$, we have 
$P \in \oo$ or 
$P \in \mathfrak{a}$. 
If $P \in \oo$ we denote by $r_{i}$ such a prime. 
If $P \in \mathfrak{a}$ we denote by $r^{\prime}_{i}$ such a prime. 

We take the element $n \in l_{1}\cap l_{2}$ and 
perform a prime factorization, 
$n=p_{1}\cdots p^{\prime}_{1}\cdots q_{1}\cdots r_{1}\cdots 
r^{\prime}_{1}\cdots$. 
Then, $p_{1}\cdots$, $q_{1}\cdots$ and $r_{1}\cdots$ 
correspond to principal ideals. 
So, if the number of the primes $p^{\prime}$ and $r^{\prime}$ 
is even then $n\in l_{1}$ and 
if the number of the primes $p^{\prime}$ and $r^{\prime}$ 
is odd then $n\in l_{2}$. 
This completes the proof of Lemma \ref{lem:coprime}. 
\end{proof}

\begin{lem}\label{lem:num2}
Let $d$ be one of the elements in $\{$$5$, $6$, $10$, $13$, $15$, 
$22$, $35$, $37$, $51$, $58$, $91$, $115$, $123$, $187$, $235$, 
$267$, $403$, $427$$\}$. 
We set $a^{\prime}(m)=a(m)/2$ for all $m$. Then, 
\begin{eqnarray*}
a^{\prime}(p^{e})=\left\{
\begin{array}{lll}
e+1 &{ if}\ \left(d_{K}/p\right)=1, \\
(1+(-1)^e)/2 &{ if}\ \left(d_{K}/p\right)=-1, \\
1 &{ if}\ p\ \vert\ d_{K}. 
\end{array} 
\right.
\end{eqnarray*}
\end{lem}
\begin{proof}
Because of the equation (\ref{eqn:multi}), 
$a^{\prime}(m)$ is the number of integral ideals of $K$ of norm $m$. 
Therefore, it can be proved by Proposition \ref{prop:NUM}. 
\end{proof}

\begin{lem}\label{lem:non1}
Let $p$ be a prime number such that $(d_{K}/p)=1$. Then, $b(p)\neq 0$. 
Moreover, if $p\neq d$ then $b(p)\neq \pm p$. 
\end{lem}
\begin{proof}
We remark that by Corollary \ref{cor:HeckeTheta}, 
$c_{1}\Theta _{L_{\oo},P_2}(q)+c_{2}\Theta _{L_{\mathfrak{a}},P_2}(q)
=\Psi_{K,\Lambda}(q)$. 
So, the numbers $b(m)$ are the coefficients of $\Psi_{K,\Lambda}(q)$. 

We set $N(J)=p$. 
When $J$ is a principal ideal, it can be proved 
by the similar method in Lemma \ref{lem:non0}. 
So, we assume that $J$ is nonprincipal. 


We list the smallest prime number $m$ such that $m\ \vert\ d_K$ 
and $m\in \{N(I)\mid I\in \mathfrak{a}\}$, 
and the values $b(m)$ are in Table \ref{Tab:min L_a}. 
\begin{table}[thb]
\caption{Values of $m$ and $b(m)$}
\label{Tab:min L_a}
\begin{center}
{\footnotesize
\begin{tabular}{c||lllllllll}
\noalign{\hrule height0.8pt}\hline
$-d$               & \MC{1}{c}{$-5$} & \MC{1}{c}{$-6$} & \MC{1}{c}{$-10$} 
                   & \MC{1}{c}{$-13$} & \MC{1}{c}{$-15$} & \MC{1}{c}{$-22$}
                   & \MC{1}{c}{$-35$} & \MC{1}{c}{$-37$} & \MC{1}{c}{$-51$}\\
\hline
$m$       & \MC{1}{c}{$2$} & \MC{1}{c}{$2$} &\MC{1}{c}{$2$}  
          & \MC{1}{c}{$2$} & \MC{1}{c}{$3$} &\MC{1}{c}{$2$} 
          & \MC{1}{c}{$5$} & \MC{1}{c}{$2$} &\MC{1}{c}{$3$}  \\
\hline
$b(m)$       & \MC{1}{c}{$2$} & \MC{1}{c}{$2$} &\MC{1}{c}{$2$}  
          & \MC{1}{c}{$2$} & \MC{1}{c}{$-3$} &\MC{1}{c}{$2$} 
          & \MC{1}{c}{$-5$} & \MC{1}{c}{$2$} &\MC{1}{c}{$3$}  \\
\hline
\hline
$-d$               & \MC{1}{c}{$-58$} & \MC{1}{c}{$-91$} & \MC{1}{c}{$-115$} 
                   & \MC{1}{c}{$-123$} & \MC{1}{c}{$-187$} & \MC{1}{c}{$-235$}
                   & \MC{1}{c}{$-267$} & \MC{1}{c}{$-403$} & \MC{1}{c}{$-427$}\\
\hline
$m$       & \MC{1}{c}{$2$} & \MC{1}{c}{$7$} &\MC{1}{c}{$5$}  
          & \MC{1}{c}{$3$} & \MC{1}{c}{$11$} &\MC{1}{c}{$5$} 
          & \MC{1}{c}{$3$} & \MC{1}{c}{$13$} &\MC{1}{c}{$7$}  \\
\hline
$b(m)$       & \MC{1}{c}{$2$} & \MC{1}{c}{$-7$} &\MC{1}{c}{$-5$}  
          & \MC{1}{c}{$3$} & \MC{1}{c}{$-11$} &\MC{1}{c}{$5$} 
          & \MC{1}{c}{$3$} & \MC{1}{c}{$-13$} &\MC{1}{c}{$7$}  \\
\hline
\noalign{\hrule height0.8pt}
   \end{tabular}
}
\end{center}
\end{table}
First, we assume that $-d\equiv 2$ or $3 \pmod 4$. 
If $N(J)$ is equal to $p$ then by Lemma \ref{lem:Takagi} 
\begin{eqnarray*}
mp=a^{2}+db^{2}. 
\end{eqnarray*}
Because of the definition of $\Psi_{K,\Lambda}(q)$, 
\begin{eqnarray*}
b(mp)= (a+b\sqrt{-d})^2 + (a-b\sqrt{-d})^2=2(a^{2}-db^{2}). 
\end{eqnarray*}
Since $b(mp)=b(m)b(p)$ and the value of $b(m)$ 
in Table \ref{Tab:min L_a}, we have $b(p)=a^{2}-db^{2}$. 
If $b(p)= 0$ then $a^{2}=db^{2}$. This is a contradiction. 
Assume that $b(p)= \pm p$. Then, 
\begin{eqnarray*}
a^{2}-db^{2}=\pm \frac{a^{2}+db^{2}}{2}, 
\end{eqnarray*}
that is, $a^{2}=3db^{2}$ or $3a^{2}=db^{2}$. 
This is a contradiction. 

Next, we assume that $-d\equiv 1 \pmod 4$. 
If $N(J)$ is equal to $p$ then by Lemma \ref{lem:Takagi} 
there exist $a$, $b\in\ZZ$ such that 
\begin{eqnarray*}
mp=a^{2}+db^{2}\quad or \quad mp=\frac{a^{2}+db^{2}}{4}. 
\end{eqnarray*}
Because of the definition of $\Psi_{K,\Lambda}(q)$, 
\begin{eqnarray*}
b(mp)= (a+b\sqrt{-d})^2 + (a-b\sqrt{-d})^2=2(a^{2}-db^{2}). 
\end{eqnarray*}
or 
\begin{eqnarray*}
b(mp)= \Big(\frac{a+b\sqrt{-d}}{2}\Big)^2 + 
\Big(\frac{a-b\sqrt{-d}}{2}\Big)^2=\frac{a^{2}-db^{2}}{2}. 
\end{eqnarray*}
Since $b(mp)=b(m)b(p)$ and the value of $b(m)$ 
in Table \ref{Tab:min L_a}, we have $b(p)=2/b(m)\times(a^{2}-db^{2})$ or 
$b(p)=1/b(m)\times(a^{2}-db^{2})/2$. 
If $b(p)= 0$ then $a^{2}=db^{2}$. This is a contradiction. 
Assume that $b(p)= \pm p$. Then, 
\begin{eqnarray*}
\frac{2(a^{2}-db^{2})}{b(m)}=\pm \frac{a^{2}+db^{2}}{m}, 
\end{eqnarray*}
or 
\begin{eqnarray*}
\frac{a^{2}-db^{2}}{2b(m)}=\pm \frac{a^{2}+db^{2}}{4m}, 
\end{eqnarray*}
that is, $a^{2}=3db^{2}$ or $3a^{2}=db^{2}$ 
since $m=\pm b(m)$ for $-d\equiv 1\pmod 4$. 
This is a contradiction. 
\end{proof}
\hspace{-17pt}{\it Proof of Theorem \ref{thm:classnumber=2}. }
Because of Lemma \ref{lem:coprime}, 
it is enough to show that $b(m)\neq 0$ when $(L_{\oo})_{m}\neq \emptyset$ 
or $(L_{\mathfrak{a}})_{m}\neq \emptyset$. 

By Theorem \ref{thm:ono}, 
$c_{1}\Theta _{L_{\oo},P_2}+c_{2}\Theta _{L_{\mathfrak{a}},P_2}$ is a normalized Hecke 
eigenform. So, We assume that $m$ is a power of prime, 
if not we could apply the equation (\ref{eqn:mul}).
We will divide into the three cases. 
\begin{enumerate}
\item 
[(i)] Case $m=p^{\alpha}$ and $p\ \vert\ d_{K}$: \\
By $a(m)=2$ and (\ref{ine:Fisher}), 
the shells $(L)_{m}$ are not spherical $2$-designs. 
Hence, $b(m)\neq 0$. 

\item 
[(ii)] Case $m=p^{\alpha}$ and $\left(d_{K}/p\right)=-1$: \\
By Lemma (\ref{lem:num}), 
\begin{eqnarray*}
a(p^{n})= 
\left\{
\begin{array}{llll}
0  \quad &{\rm if\ }n\ {\rm is\ odd}, \\
2  \quad &{\rm if\ }n\ {\rm is\ even}. 
\end{array} 
\right.
\end{eqnarray*}
By $a(m)=2$ and (\ref{ine:Fisher}), when $n$ is even, 
the shells $(L)_{m}$ are not spherical $2$-designs. 
Hence, $b(m)\neq 0$. 

\item 
[(iii)] Case $m=p^{\alpha}$ and $\left(d_{K}/p\right)=1$: \\
By Proposition \ref{prop:Lehmer} and Lemma \ref{lem:non1}, 
$b(m)\neq 0$. 
This completes the proof of Theorem \ref{thm:classnumber=2}. \e
\end{enumerate}
\section{The case of $\vert \Cl_{K}\vert =3$}
In the previous sections, we studied the cases of class 
number $h=\vert \Cl_{K}\vert $ is either $1$ or $2$. 
However, it seems that the situation is somewhat different 
for the cases of class numbers $h \geq 3$. 
In this section, we discuss briefly how it is different, 
by considering the case of $d=23$ $(h=3)$. 

We first remark that one reason of success for the cases $h=1$ and $h=2$ 
is that the coefficients $a(m)$ of the 
Hecke eigenform $\Psi_{K, \Lambda}$ are all integers. 
Therefore, by the formula (\ref{lem:2})
$z=a(p)/p$ is a rational number (and since it is an algebraic 
integer), and so it must be a rational integer. 
It seems that this situation is no more true in general 
for the cases of $h\geq 3$. 
We will give more details information, 
concentrating the special (and typical) case of $d=23$. 

We denoted by $\oo$, $\mathfrak{a}_{1}$ and $\mathfrak{a}_{2}$ 
the ideal classes. 
The corresponding quadratic forms are $x^2+xy+6y^2$, 
$2x^2-xy+3y^2$ and 
$2x^2+xy+3y^2$, namely, 
$L_{\oo}=\langle (1,0), (1/2, \sqrt{23}/2)\rangle$, 
$L_{\mathfrak{a}_{1}}=\langle (2,0), (1/2, \sqrt{23}/2)\rangle$ and 
$L_{\mathfrak{a}_{2}}=\langle (2,0), (-1/2, \sqrt{23}/2)\rangle$, respectively. 
We give the weighed theta series of those ideal lattices. 
We set $P_{1}=x^2-y^2$ and $P_{2}=xy$ in this section. \\ \\
{\footnotesize $\Theta_{L_{\oo}}=1+2 q+2 q^4+4 q^6+4 q^8+2 q^9+4 q^{12}+2 q^{16}+4 q^{18}+2 q^{23}+4 q^{24}+2 q^{25}+4 q^{26}+4 q^{27}+4 q^{32}+6 q^{36}+4 q^{39}+8 q^{48}+2 q^{49}+4 q^{52}+4 q^{54}+4 q^{58}+4 q^{59}+4 q^{62}+6 q^{64}+8 q^{72}+4 q^{78}+2 q^{81}+4 q^{82}+4 q^{87}+2 q^{92}+4 q^{93}+4 q^{94}+8 q^{96}+2 q^{100}+O[q]^{101}$\\ \\
$\frac{1}{2}\times\Theta_{L_{\oo}, P_{1}}=q+4 q^4-11 q^6-7 q^8+9 q^9+q^{12}+16 q^{16}+13 q^{18}-23 q^{23}-44 q^{24}+25 q^{25}+29 q^{26}-38 q^{27}-28 q^{32}+85 q^{36}-14 q^{39}+77 q^{48}+49 q^{49}-103 q^{52}-99 q^{54}-91 q^{58}+26 q^{59}+101 q^{62}-15 q^{64}-11 q^{72}+133 q^{78}+81 q^{81}-43 q^{82}+82 q^{87}-92 q^{92}-182 q^{93}-19 q^{94}-7 q^{96}+100 q^{100}+O[q]^{101}$\\ \\
$\Theta_{L_{\oo}, P_{2}}=0$\\ \\
$\Theta_{L_{\mathfrak{a}_{1}}}=1+2 q^2+2 q^3+2 q^4+2 q^6+2 q^8+2 q^9+4 q^{12}+2 q^{13}+4 q^{16}+4 q^{18}+6 q^{24}+2 q^{26}+2 q^{27}+2 q^{29}+2 q^{31}+4 q^{32}+6 q^{36}+2 q^{39}+2 q^{41}+2 q^{46}+2 q^{47}+6 q^{48}+2 q^{50}+4 q^{52}+6 q^{54}+2 q^{58}+2 q^{62}+4 q^{64}+2 q^{69}+2 q^{71}+8 q^{72}+2 q^{73}+2 q^{75}+6 q^{78}+4 q^{81}+2 q^{82}+2 q^{87}+2 q^{92}+2 q^{93}+2 q^{94}+8 q^{96}+2 q^{98}+2 q^{100}+O[q]^{101}$\\ \\
$2\times\Theta_{L_{\mathfrak{a}_{1}}, P_{1}}=8 q^2-11 q^3-7 q^4+q^6+32 q^8+13 q^9-88 q^{12}+29 q^{13}-56 q^{16}+121 q^{18}+81 q^{24}-103 q^{26}-99 q^{27}-91 q^{29}+101 q^{31}+49 q^{32}+41 q^{36}+133 q^{39}-43 q^{41}-184 q^{46}-19 q^{47}-183 q^{48}+200 q^{50}+232 q^{52}-295 q^{54}+209 q^{58}+41 q^{62}-224 q^{64}+253 q^{69}+77 q^{71}+393 q^{72}-283 q^{73}-275 q^{75}-375 q^{78}+418 q^{81}-247 q^{82}-227 q^{87}+161 q^{92}-203 q^{93}+353 q^{94}+616 q^{96}+392 q^{98}-175 q^{100}+O[q]^{101}$\\ \\
$\frac{4}{\sqrt {23}}\times\Theta_{L_{\mathfrak{a}_{1}}, P_{2}}=q^3-3 q^4+5 q^6-7 q^9+9 q^{13}-11 q^{18}+13 q^{24}-3 q^{26}+9 q^{27}-15 q^{29}-15 q^{31}+21 q^{32}-27 q^{36}+17 q^{39}+33 q^{41}-39 q^{47}-19 q^{48}+45 q^{54}+21 q^{58}-51 q^{62}-23 q^{69}+57 q^{71}+5 q^{72}-15 q^{73}+25 q^{75}-35 q^{78}-38 q^{81}+45 q^{82}-55 q^{87}+69 q^{92}+65 q^{93}-27 q^{94}-75 q^{100}+O[q]^{101}$\\ \\
$\Theta_{L_{\mathfrak{a}_{2}}}=1+2 q^2+2 q^3+2 q^4+2 q^6+2 q^8+2 q^9+4 q^{12}+2 q^{13}+4 q^{16}+4 q^{18}+6 q^{24}+2 q^{26}+2 q^{27}+2 q^{29}+2 q^{31}+4 q^{32}+6 q^{36}+2 q^{39}+2 q^{41}+2 q^{46}+2 q^{47}+6 q^{48}+2 q^{50}+4 q^{52}+6 q^{54}+2 q^{58}+2 q^{62}+4 q^{64}+2 q^{69}+2 q^{71}+8 q^{72}+2 q^{73}+2 q^{75}+6 q^{78}+4 q^{81}+2 q^{82}+2 q^{87}+2 q^{92}+2 q^{93}+2 q^{94}+8 q^{96}+2 q^{98}+2 q^{100}+O[q]^{101}$\\ \\
$2\times\Theta_{L_{\mathfrak{a}_{2}}, P_{1}}=8 q^2-11 q^3-7 q^4+q^6+32 q^8+13 q^9-88 q^{12}+29 q^{13}-56 q^{16}+121 q^{18}+81 q^{24}-103 q^{26}-99 q^{27}-91 q^{29}+101 q^{31}+49 q^{32}+41 q^{36}+133 q^{39}-43 q^{41}-184 q^{46}-19 q^{47}-183 q^{48}+200 q^{50}+232 q^{52}-295 q^{54}+209 q^{58}+41 q^{62}-224 q^{64}+253 q^{69}+77 q^{71}+393 q^{72}-283 q^{73}-275 q^{75}-375 q^{78}+418 q^{81}-247 q^{82}-227 q^{87}+161 q^{92}-203 q^{93}+353 q^{94}+616 q^{96}+392 q^{98}-175 q^{100}+O[q]^{101}$\\ \\
$\frac{4}{\sqrt {23}}\times\Theta_{L_{\mathfrak{a}_{2}}, P_{2}}=-q^3+3 q^4-5 q^6+7 q^9-9 q^{13}+11 q^{18}-13 q^{24}+3 q^{26}-9 q^{27}+15 q^{29}+15 q^{31}-21 q^{32}+27 q^{36}-17 q^{39}-33 q^{41}+39 q^{47}+19 q^{48}-45 q^{54}-21 q^{58}+51 q^{62}+23 q^{69}-57 q^{71}-5 q^{72}+15 q^{73}-25 q^{75}+35 q^{78}+38 q^{81}-45 q^{82}+55 q^{87}-69 q^{92}-65 q^{93}+27 q^{94}+75 q^{100}+O[q]^{101}$\\
}

We calculate the Hecke character of weight $3$ and modulus $(1)$, 
i.e, we
calculate $\Psi_{K, \Lambda}=\sum_{m\geq 1}a(m)q^{m}$, 
where $\Lambda =(1)$ and $k = 3$. 
When $A$ of norm $m$ is a nonprincipal ideal, 
$A^3$ is a principal ideal. 
Then we set $\phi(A)^3 = \phi(A^3)$. 
For example, 
$(2, -1/2 +\sqrt{-23}/2)^3=(-3/2 -\sqrt{-23}/2)$. 
Because of 
\[
\phi \Big(\Big(\frac{-3 -\sqrt{-23}}{2}\Big)\Big)=
\Big(\frac{-3 -\sqrt{-23}}{2}\Big)^2=
\frac{-7+3\sqrt{-23}}{2}, 
\]
$\phi ((2, -1/2 +\sqrt{-23}/2))$ is 
one of the roots of 
\begin{eqnarray}\label{eqn:sec6}
x^3-\Big(\frac{-7+3\sqrt{-23}}{2}\Big)=0. 
\end{eqnarray}
We denote by $\alpha_1$, $\alpha_2$ and $\alpha_3$ 
the roots of equation (\ref{eqn:sec6}), namely, 
$\alpha_1\sim -1.86272+0.728188i$, $\alpha_2\sim 0.300733-1.97726i$ 
and $\alpha_3\sim 1.56199+1.24907i$, respectively. 
Then, $\phi ((2, -1/2 +\sqrt{-23}/2))$ is one of $\alpha_1$, $\alpha_2$ 
or $\alpha_3$. 
(Actually there
are three different Hecke characters in this case.) 
First let us set $\phi ((2, -1/2 +\sqrt{-23}/2))= \alpha_1$. 
By the equation $(2, -1/2 +\sqrt{-23}/2)\times(2, 1/2 +\sqrt{-23}/2)=(2)$, 
\[
\phi \Big(\Big(2, \frac{-1 +\sqrt{-23}}{2}\Big)\Big)\times
\phi \Big(\Big(2, \frac{1 +\sqrt{-23}}{2}\Big)\Big)=\phi((2)). 
\]
We get 
\[
\alpha_1 \times
\phi \Big(\Big(2, \frac{1 +\sqrt{-23}}{2}\Big)\Big)=4, 
\]
hence, $\phi ((2, 1/2 +\sqrt{-23}/2))= 4/\alpha_1$. 
So, 
\[
a(2)=\phi \Big(\Big(2, \frac{-1 +\sqrt{-23}}{2}\Big)\Big)+
\phi \Big(\Big(2, \frac{1 +\sqrt{-23}}{2}\Big)\Big)=\alpha_1+4/\alpha_1. 
\]

By the equation $(2, -1/2 +\sqrt{-23}/2)\times(3, 1/2 -\sqrt{-23}/2)
=(1/2 -\sqrt{-23}/2)$, 
\[
\phi \Big(\Big(2, \frac{-1 +\sqrt{-23}}{2}\Big)\Big)\times
\phi \Big(\Big(3, \frac{1 -\sqrt{-23}}{2}\Big)\Big)
=\phi\Big(\Big(\frac{1 -\sqrt{-23}}{2}\Big)\Big). 
\]
We get 
\[
\alpha_1\times
\phi \Big(\Big(3, \frac{1 -\sqrt{-23}}{2}\Big)\Big)
=\Big(\frac{1 -\sqrt{-23}}{2}\Big)^2=\frac{-11- \sqrt{-23}}{2}, 
\]
hence, $\phi ((3, 1/2 -\sqrt{-23}/2))= 
(-11- \sqrt{-23})/2\times 1/\alpha_1$. 
Similarly, $\phi ((3, -1/2 -\sqrt{-23}/2))= 
(-11+ \sqrt{-23})/2\times \alpha_1/(\alpha_1^2+4)$.
So, 
\begin{eqnarray*}
\displaystyle a(3)&=&\phi \Big(\Big(3, \frac{1 -\sqrt{-23}}{2}\Big)\Big)+
\phi \Big(\Big(3, \frac{-1 -\sqrt{-23}}{2}\Big)\Big)\\
&=&\frac{-11- \sqrt{-23}}{2}\times\frac{1}{\alpha_1}+
\frac{-11+ \sqrt{-23}}{2}\times \frac{\alpha_1}{\alpha_1^2+4}. 
\end{eqnarray*}
We recall $\alpha_1\sim -1.86272+0.728188i$. 
Then, we obtain 
\[\Psi_{K, \Lambda}^{(1)}=q-3.72545 q^2+4.24943 q^3+\cdots.
\] 
Actually, it is possible to continue this calculation, but we 
need the information on the basis of all the ideals, which 
is rather complicated. So, we determine the Hecke eigenforms 
$\Psi_{K, \Lambda}^{(i)}$ by a different method. 
By computer calculation (using ``Sage" \cite{Sage}), 
we know that the space of the modular forms of weight $3$ where 
$\Psi_{K, \Lambda}$ belongs is of dimension $3$. 
We can calculate the basis 
of this modular form explicitly, and the three basis elements are 
of the form: 
\begin{eqnarray*}
&&q + 4q^4 - 11q^6 - 7q^8 + 9q^9 +\cdots , \\
&&q^2 - 5q^4 + 7q^6 + 4q^8 - 8q^9 +\cdots , \\
&&q^3 - 3q^4 + 5q^6 - 7q^9 +\cdots . 
\end{eqnarray*}
On the other hand, 
because of Theorems \ref{thm:ono} and \ref{thm:Miyake}, 
$\Theta_{L_{\oo},P_1}$, 
$\Theta_{L_{\mathfrak{a}_1}, P_1}$ and 
$\Theta_{L_{\mathfrak{a}_2}, P_2}$ are in the same space 
of Hecke eigenforms $\Psi_{K, \Lambda}^{(i)}$. 
Therefore, comparing the first three coefficients of the 
following equation:
\begin{eqnarray*}
\Psi_{K, \Lambda}^{(1)}(q)=\frac{1}{2}\Theta_{L_{\oo},P}(q)+
a2\Theta_{L_{\mathfrak{a}_1},P}(q)+
b\frac{4}{\sqrt{23}}\Theta_{L_{\mathfrak{a}_2},P}(q), 
\end{eqnarray*}
we can find numbers $a$ and $b$ as follows: 
\begin{eqnarray*}
(a, b)=
\left\{
\begin{array}{l}
(A_{1}, B_{2}), \\
(A_{2}, B_{1}), \\
(A_{3}, B_{3}), 
\end{array}
\right.
\end{eqnarray*}
where $A_{1}$, $A_{2}$ and $A_{3}$ are the elements defined by 
\begin{eqnarray*}
\{x\mid 512x^3-96x+7=0\}\hspace{180pt}\\
=\{A_1=-0.465681, A_2=0.0751832, A_2=0.390498\}, 
\end{eqnarray*}
respectively, 
and $B_{1}$, $B_{2}$ and $B_{3}$ are the elements defined by 
\begin{eqnarray*}
\{x\mid 512x^3-2208x+1587=0\}\hspace{150pt}\\
=\{B_1=-2.37065, B_2=0.873067, B_3=1.49759\},\hspace{17pt} 
\end{eqnarray*}
respectively.

In this way, we can calculate the Hecke eigenforms $\Psi_{K, \Lambda}^{(i)}$. 
Namely, \\
{\footnotesize $\Psi_{K, \Lambda}^{(1)}=q-3.72545 q^2+4.24943 q^3+9.87897 q^4-15.831 q^6-21.9018 q^8+9.05761 q^9+41.9799 q^{12}-21.3624 q^{13}+42.0781 q^{16}-33.7437 q^{18}-23 q^{23}-93.07 q^{24}+25 q^{25}+79.5844 q^{26}+0.244826 q^{27}+55.473 q^{29}-33.9378 q^{31}-69.1528 q^{32}+89.4799 q^{36}-90.7777 q^{39}-8.78692 q^{41}+85.6853 q^{46}+42.8975 q^{47}+178.808 q^{48}+49 q^{49}-93.1362 q^{50}+O[q]^{51}$. \\ \\
$\Psi_{K, \Lambda}^{(2)}=q+0.601466 q^2+1.54364 q^3-3.63824 q^4+0.928445 q^6-4.59414 q^8-6.61718 q^9-5.61612 q^{12}+23.5162 q^{13}+11.7897 q^{16}-3.98001 q^{18}-23 q^{23}-7.09168 q^{24}+25 q^{25}+14.1442 q^{26}-24.1073 q^{27}-42.4015 q^{29}-27.9663 q^{31}+25.4677 q^{32}+24.0749 q^{36}+36.3005 q^{39}+74.9986 q^{41}-13.8337 q^{46}-93.8839 q^{47}+18.1991 q^{48}+49 q^{49}+15.0366 q^{50}+O[q]^{51}$. 
\\ \\
$\Psi_{K, \Lambda}^{(3)}=q+3.12398 q^2-5.79306 q^3+5.75927 q^4-18.0974 q^6+5.49593 q^8+24.5596 q^9-33.3638 q^{12}-2.15383 q^{13}-5.86788 q^{16}+76.7237 q^{18}-23 q^{23}-31.8383 q^{24}+25 q^{25}-6.72853 q^{26}-90.1376 q^{27}-13.0715 q^{29}+61.9041 q^{31}-40.3149 q^{32}+141.445 q^{36}+12.4773 q^{39}-66.2117 q^{41}-71.8516 q^{46}+50.9864 q^{47}+33.993 q^{48}+49 q^{49}+78.0996 q^{50}+O[q]^{51}$. \\ \\}
The coefficients $a(m)$ for this case are far from integers. 
In fact they are not elements in a cyclotomic number field in general. 
So, it seems difficult to use the Hecke eigenforms obtained this way 
to apply for the case of the class number $3$ or more in general. 
Some new additional ideas will be needed to treat the case of 
$d=23$ or more generally the cases of class numbers $h\geq 3$. 
We have included the presentation of the results (although they are  
not conclusive) for $d=23$, 
hoping that it might help the reader for the future study on this topic.

\begin{rem}
We remark that the coefficients of $\Psi_{K, \Lambda}^{(i)}$ 
in above calculator results are 
not exact values but approximate values. 
\end{rem}





\begin{table}[thb]
\caption{Integral ideals of small norm of $d=23$}
\label{Tab:d=23}
\begin{center}
{\footnotesize
\begin{tabular}{c||c} 
\noalign{\hrule height0.8pt}
\hline
$N(A)$ & $A$: ideal  \\ \hline
$1$& $(1)$ \\ \hline
$2$& $(2, -1/2 +\sqrt{-23}/2)$ \\ 
   & $(2, 1/2 +\sqrt{-23}/2)$ \\ \hline
$3$& $(3, 1/2 -\sqrt{-23}/2)$ \\ 
   & $(3, -1/2 -\sqrt{-23}/2)$ \\ \hline
$4$& $(4, 3/2 +\sqrt{-23}/2)$ \\ 
   & $(2)$ \\ 
   & $(4, 5/2 +\sqrt{-23}/2)$ \\ \hline
$5$& $-$ \\ \hline
\noalign{\hrule height0.8pt}
\end{tabular}
\hspace{10pt}
\begin{tabular}{c||c} 
\noalign{\hrule height0.8pt}
\hline
$N(A)$ & $A$: ideal  \\ \hline
$6$& $(1/2 -\sqrt{-23}/2)$ \\ 
   & $(6, 5/2 +\sqrt{-23}/2)$ \\ 
   & $(6, 7/2 +\sqrt{-23}/2)$ \\ 
   & $(1/2 +\sqrt{-23}/2)$ \\ \hline
$7$& $-$ \\ \hline
$8$& $(-3/2 -\sqrt{-23}/2)$ \\ 
   & $(4, -1 +\sqrt{-23})$ \\ 
   & $(4, 1 +\sqrt{-23})$ \\ 
   & $(-3/2 +\sqrt{-23}/2)$ \\ \hline
$9$& $(9, 11/2 +\sqrt{-23}/2)$ \\ 
   & $(3)$ \\ 
   & $(9, 7/2 +\sqrt{-23}/2)$ \\ \hline
$10$& $-$ \\ \hline
\noalign{\hrule height0.8pt}
\end{tabular}
}
\end{center}
\end{table}

\newpage
\section{Concluding Remarks}
\begin{enumerate}
\item 
[(1)]
In this paper, we use the mathematics software ``Sage" \cite{Sage}. 
In particular, The results in Tables \ref{Tab:d=2,d=5} and \ref{Tab:d=23} are 
compute by ``Sage" using the command ``K.ideals\_of\_bdd\_norm()". 
We remark that this command do not always give a $\ZZ$-basis of ideal. 
We must make sure the command ``(ideal).basis()". 

\item 
[(2)]
In Appendix C, we list theta series of lattices 
obtained from $\QQ(\sqrt{-5})$. 
The other cases are listed in one of the auther's website \cite{miezaki}. 

\item 
[(3)]
In the previous paper \cite{Toy-BM}, we studied the spherical designs in 
the nonempty shells of the $\ZZ^{2}$-lattice 
and $A_{2}$-lattice. 
The results state that any shells in the $\ZZ^{2}$-lattice 
(resp.\ $A_{2}$-lattice) are 
spherical $2$-design (resp.\ $4$-design). However, 
the nonempty shells in the $\ZZ^{2}$-lattice 
(resp.\ $A_{2}$-lattice) are not 
spherical $4$-design (resp.\ $6$-design). 
It is interesting to note that no spherical $6$-design among 
the nonempty shells 
of any Euclidean lattice of $2$-dimension is known. 
It is an interesting open problem 
to prove or disprove whether these exists any $6$-design 
which is a shell of a Euclidean lattice of $2$-dimension. 

Responding to the authors' request, Junichi Shigezumi performed computer 
calculations to determine whether there are spherical $t$-designs for bigger $t$, in the $2$-
and $3$-dimensional cases. His calculation shows that among the nonempty shells of 
integral lattices in $2$-dimension (with relatively small discriminant 
and small norms), there are only $4$-designs. 
That is, no $6$-designs were found. (So far, all examples of such $4$-designs 
are the union of vertices of regular $6$-gons, although they are 
the nonempty shells 
of many different lattices). 
In the $3$-dimensional case, 
all examples obtained are only $2$-designs. 
No $4$-designs which are shells 
of a lattice were found. It is an interesting open problem 
whether this is true in general for the dimensions $2$ and $3$. 
Moreover, it is interesting to note that 
no spherical $12$-design among the nonempty shells 
of any Euclidean lattice (of any dimension) is known. 
It is also an interesting open problem 
to prove or disprove whether these exists any $12$-design 
which is a shell of a Euclidean lattice. 

Finally, we state the following conjecture 
for the $2$-dimensional lattices. 
\begin{conj}
Let $L$ be the Euclidean lattice of $2$-dimension, 
whose quadratic form is $ax^2+bxy+cy^2$. 
\begin{enumerate}
\item 
[{\rm (i)}]
 Assume that $b^2-4ac=$$($Integer$)^2$$\times (-3)$. 
Then, all the nonempty shells of $L$ are 
not spherical $6$-designs and 
some of the nonempty shells of $L$ are spherical $4$-designs. Moreover, if all the nonempty shells of $L$ are spherical $4$-designs then\\ 
$b^2-4ac= -3$, that is, $A_{2}$-lattice. 

\item 
[{\rm (ii)}]
 Assume that $b^2-4ac=$$($Integer$)^2$$\times (-4)$. 
Then, all the nonempty shells of $L$ are 
not spherical $4$-designs and 
some of the nonempty shells of $L$ are spherical $2$-designs. Moreover, if all the nonempty shells of $L$ are spherical $2$-designs then\\ 
$b^2-4ac= -4$, that is, $\ZZ^{2}$-lattice. 

\item 
[{\rm (iii)}]
 Otherwise, all the nonempty shells of $L$ are not spherical $2$-designs. 
\end{enumerate}
\end{conj}
\end{enumerate}

\bigskip
\noindent
{\bf Acknowledgment.}
The authors thank Masao Koike for informing us that our 
previous results in \cite{Toy-BM} can be interpreted in terms of 
the cusp forms attached to $L$-functions with a Hecke character 
of CM fields, i.e., the imaginary quadratic fields 
$\QQ(\sqrt{-1})$ and $\QQ(\sqrt{-3})$, and in particular for bringing 
our attention to Theorem 1.31 in \cite{Web}. 
The authors also thank Junichi Shigezumi 
for his helpful discussions and computations on this research. 
The second author is supported by JSPS Research Fellowship. 
\begin{appendix}

\section{The case of $\vert \Cl_{K}\vert =1$}
\begin{table}[thb]
\caption{$\vert \Cl_{K}\vert =1$}
\label{Tab:Cl=1}
\begin{center}
{\footnotesize
\begin{tabular}{c|c|c|c} 
\noalign{\hrule height0.8pt}
$-d$ & $-d \pmod 4$ & $d_{K}$ & $L_{\oo}$ \\ \hline 
\hline
$-1$ & $3$ & $-2^2$ & $[1, \sqrt{-1}]$ \\ \hline

$-2$ & $2$ & $-2^3$ & $[1, \sqrt{-2}]$ \\ \hline

$-3$ & $1$ & $-3$ & $[1, (1+\sqrt{-3})/2]$ \\ \hline

$-7$ & $1$ & $-7$ & $[1, (1+\sqrt{-7})/2]$ \\ \hline

$-11$ & $1$ & $-11$ & $[1, (1+\sqrt{-11})/2]$ \\ \hline

$-19$ & $1$ & $-19$ & $[1, (1+\sqrt{-19})/2]$ \\ \hline

$-43$ & $1$ & $-43$ & $[1, (1+\sqrt{-43})/2]$ \\ \hline

$-67$ & $1$ & $-67$ & $[1, (1+\sqrt{-67})/2]$ \\ \hline

$-163$ & $1$ & $-163$ & $[1, (1+\sqrt{-163})/2]$ \\ \hline
\noalign{\hrule height0.8pt}
\end{tabular}
\hspace{10pt}
}
\end{center}
\end{table}
\newpage
\section{The case of $\vert \Cl_{K}\vert =2$}
\begin{table}[thb2]
\caption{$\vert \Cl_{K}\vert =2$}
\label{Tab:Cl=2}
\begin{center}
{\footnotesize
\begin{tabular}{c|c|c|c|c} 
\noalign{\hrule height0.8pt}
\hline
$-d$ & $-d \pmod 4$ & $d_{K}$ & $L_{\oo}$ & $L_{\mathfrak{a}}$ \\ \hline \hline
$-5$ & $3$ & $-2^2\times 5$ & $[1, \sqrt{-5}]$ 
& $[2, 1+\sqrt{-5}]$  \\ \hline

$-6$ & $2$ & $-2^3\times 3$ & $[1, \sqrt{-6}]$ 
& $[2, \sqrt{-6}]$  \\ \hline

$-10$ & $2$ & $-2^3\times 5$ & $[1, \sqrt{-10}]$ 
& $[2, \sqrt{-10}]$  \\ \hline

$-13$ & $3$ & $-2^2\times 13$ & $[1, \sqrt{-13}]$ 
& $[2, 1+\sqrt{-13}]$  \\ \hline

$-15$ & $1$ & $-3\times 5$ & $[1, (1+\sqrt{-15})/2]$ 
& $[2, (1+\sqrt{-15})/2]$  \\ \hline

$-22$ & $2$ & $-2^3\times 11$ & $[1, \sqrt{-22}]$ 
& $[2, \sqrt{-22}]$  \\ \hline

$-35$ & $1$ & $-5\times 7$ & $[1, (1+\sqrt{-35})/2]$ 
& $[3, (1+\sqrt{-35})/2]$  \\ \hline

$-37$ & $3$ & $-2^2\times 37$ & $[1, \sqrt{-37}]$ 
& $[2, 1+\sqrt{-37}]$  \\ \hline

$-51$ & $1$ & $-3\times 17$ & $[1, (1+\sqrt{-51})/2]$ 
& $[3, (3+\sqrt{-51})/2]$  \\ \hline

$-58$ & $2$ & $-2^3\times 29$ & $[1, \sqrt{-58}]$ 
& $[2, \sqrt{-58}]$  \\ \hline

$-91$ & $1$ & $-7\times 13$ & $[1, (1+\sqrt{-91})/2]$ 
& $[5, (3+\sqrt{-91})/2]$  \\ \hline

$-115$ & $1$ & $-5\times 23$ & $[1, (1+\sqrt{-115})/2]$ 
& $[5, (5+\sqrt{-115})/2]$  \\ \hline

$-123$ & $1$ & $-3\times 41$ & $[1, (1+\sqrt{-123})/2]$ 
& $[3, (3+\sqrt{-123})/2]$  \\ \hline

$-187$ & $1$ & $-11\times 17$ & $[1, (1+\sqrt{-187})/2]$ 
& $[7, (3+\sqrt{-187})/2]$  \\ \hline

$-235$ & $1$ & $-5\times 47$ & $[1, (1+\sqrt{-235})/2]$ 
& $[5, (5+\sqrt{-235})/2]$  \\ \hline

$-267$ & $1$ & $-3\times 89$ & $[1, (1+\sqrt{-267})/2]$ 
& $[3, (3+\sqrt{-267})/2]$  \\ \hline

$-403$ & $1$ & $-13\times 31$ & $[1, (1+\sqrt{-403})/2]$ 
& $[11, (9+\sqrt{-403})/2]$  \\ \hline

$-427$ & $1$ & $-7\times 61$ & $[1, (1+\sqrt{-427})/2]$ 
& $[7, (7+\sqrt{-427})/2]$  \\ \hline
\noalign{\hrule height0.8pt}
\end{tabular}
\hspace{10pt}
}
\end{center}
\end{table}



















\newpage
\section{Theta series of $L_{\oo}$ and $L_{\mathfrak{a}}$ of $\QQ(\sqrt{-5})$}

\begin{center}
{\footnotesize
$\Theta_{L_{\oo}}=1+2 q+2 q^4+2 q^5+4 q^6+6 q^9+4 q^{14}+2 q^{16}+2 q^{20}+8 q^{21}+4 q^{24}+2 q^{25}+4 q^{29}+4 q^{30}+6 q^{36}+4 q^{41}+6 q^{45}+4 q^{46}+6 q^{49}+8 q^{54}+4 q^{56}+4 q^{61}+2 q^{64}+8 q^{69}+4 q^{70}+2 q^{80}+10 q^{81}+8 q^{84}+4 q^{86}+4 q^{89}+4 q^{94}+4 q^{96}+2 q^{100}+4 q^{101}+8 q^{105}+4 q^{109}+4 q^{116}+4 q^{120}+2 q^{121}+2 q^{125}+12 q^{126}+8 q^{129}+4 q^{134}+8 q^{141}+6 q^{144}+4 q^{145}+4 q^{149}+4 q^{150}+8 q^{161}+4 q^{164}+4 q^{166}+2 q^{169}+8 q^{174}+6 q^{180}+4 q^{181}+4 q^{184}+16 q^{189}+6 q^{196}+8 q^{201}+4 q^{205}+4 q^{206}+4 q^{214}+8 q^{216}+4 q^{224}+6 q^{225}+4 q^{229}+4 q^{230}+4 q^{241}+4 q^{244}+6 q^{245}+8 q^{246}+8 q^{249}+4 q^{254}+2 q^{256}+12 q^{261}+4 q^{269}+8 q^{270}+8 q^{276}+4 q^{280}+4 q^{281}+2 q^{289}+12 q^{294}+8 q^{301}+4 q^{305}+8 q^{309}+2 q^{320}+8 q^{321}+10 q^{324}+4 q^{326}+8 q^{329}+4 q^{334}+8 q^{336}+4 q^{344}+8 q^{345}+4 q^{349}+4 q^{350}+4 q^{356}+2 q^{361}+8 q^{366}+12 q^{369}+4 q^{376}+8 q^{381}+4 q^{384}+4 q^{389}+2 q^{400}+4 q^{401}+4 q^{404}+10 q^{405}+8 q^{406}+4 q^{409}+12 q^{414}+8 q^{420}+4 q^{421}+4 q^{430}+4 q^{436}+18 q^{441}+4 q^{445}+4 q^{446}+4 q^{449}+4 q^{454}+4 q^{461}+4 q^{464}+8 q^{469}+4 q^{470}+4 q^{480}+2 q^{484}+12 q^{486}+8 q^{489}+2 q^{500}+O[q]^{501}$\\ 
\vspace{8pt}
$\Theta_{L_{\mathfrak{a}}}=1+2 q^2+4 q^3+4 q^7+2 q^8+2 q^{10}+4 q^{12}+4 q^{15}+6 q^{18}+4 q^{23}+8 q^{27}+4 q^{28}+2 q^{32}+4 q^{35}+2 q^{40}+8 q^{42}+4 q^{43}+4 q^{47}+4 q^{48}+2 q^{50}+4 q^{58}+4 q^{60}+12 q^{63}+4 q^{67}+6 q^{72}+4 q^{75}+4 q^{82}+4 q^{83}+8 q^{87}+6 q^{90}+4 q^{92}+6 q^{98}+4 q^{103}+4 q^{107}+8 q^{108}+4 q^{112}+4 q^{115}+4 q^{122}+8 q^{123}+4 q^{127}+2 q^{128}+8 q^{135}+8 q^{138}+4 q^{140}+12 q^{147}+2 q^{160}+10 q^{162}+4 q^{163}+4 q^{167}+8 q^{168}+4 q^{172}+4 q^{175}+4 q^{178}+8 q^{183}+4 q^{188}+4 q^{192}+2 q^{200}+4 q^{202}+8 q^{203}+12 q^{207}+8 q^{210}+4 q^{215}+4 q^{218}+4 q^{223}+4 q^{227}+4 q^{232}+4 q^{235}+4 q^{240}+2 q^{242}+12 q^{243}+2 q^{250}+12 q^{252}+8 q^{258}+4 q^{263}+8 q^{267}+4 q^{268}+8 q^{282}+4 q^{283}+8 q^{287}+6 q^{288}+4 q^{290}+4 q^{298}+4 q^{300}+8 q^{303}+4 q^{307}+12 q^{315}+8 q^{322}+8 q^{327}+4 q^{328}+4 q^{332}+4 q^{335}+2 q^{338}+8 q^{343}+4 q^{347}+8 q^{348}+6 q^{360}+4 q^{362}+4 q^{363}+4 q^{367}+4 q^{368}+4 q^{375}+16 q^{378}+4 q^{383}+12 q^{387}+6 q^{392}+8 q^{402}+4 q^{410}+4 q^{412}+4 q^{415}+12 q^{423}+8 q^{427}+4 q^{428}+8 q^{432}+8 q^{435}+4 q^{443}+8 q^{447}+4 q^{448}+6 q^{450}+4 q^{458}+4 q^{460}+4 q^{463}+4 q^{467}+4 q^{482}+16 q^{483}+4 q^{487}+4 q^{488}+6 q^{490}+8 q^{492}+8 q^{498}+O[q]^{501}$

}
\end{center}
\begin{center}
{\footnotesize
$\Theta_{L_{\oo}, P}=q+4 q^4-5 q^5-8 q^6+7 q^9+8 q^{14}+16 q^{16}-20 q^{20}-16 q^{21}-32 q^{24}+25 q^{25}-22 q^{29}+40 q^{30}+28 q^{36}+62 q^{41}-35 q^{45}-88 q^{46}-33 q^{49}+16 q^{54}+32 q^{56}-58 q^{61}+64 q^{64}+176 q^{69}-40 q^{70}-80 q^{80}-95 q^{81}-64 q^{84}+152 q^{86}-142 q^{89}+8 q^{94}-128 q^{96}+100 q^{100}+122 q^{101}+80 q^{105}+38 q^{109}-88 q^{116}+160 q^{120}+121 q^{121}-125 q^{125}+56 q^{126}-304 q^{129}-232 q^{134}-16 q^{141}+112 q^{144}+110 q^{145}+278 q^{149}-200 q^{150}-176 q^{161}+248 q^{164}+152 q^{166}+169 q^{169}+176 q^{174}-140 q^{180}-358 q^{181}-352 q^{184}+32 q^{189}-132 q^{196}+464 q^{201}-310 q^{205}-88 q^{206}+248 q^{214}+64 q^{216}+128 q^{224}+175 q^{225}-262 q^{229}+440 q^{230}+302 q^{241}-232 q^{244}+165 q^{245}-496 q^{246}-304 q^{249}-472 q^{254}+256 q^{256}-154 q^{261}+38 q^{269}-80 q^{270}+704 q^{276}-160 q^{280}-418 q^{281}+289 q^{289}+264 q^{294}+304 q^{301}+290 q^{305}+176 q^{309}-320 q^{320}-496 q^{321}-380 q^{324}-328 q^{326}+16 q^{329}+488 q^{334}-256 q^{336}+608 q^{344}-880 q^{345}-22 q^{349}+200 q^{350}-568 q^{356}+361 q^{361}+464 q^{366}+434 q^{369}+32 q^{376}+944 q^{381}-512 q^{384}-202 q^{389}+400 q^{400}-478 q^{401}+488 q^{404}+475 q^{405}-176 q^{406}-802 q^{409}-616 q^{414}+320 q^{420}-778 q^{421}-760 q^{430}+152 q^{436}-231 q^{441}+710 q^{445}+872 q^{446}+398 q^{449}-712 q^{454}+842 q^{461}-352 q^{464}-464 q^{469}-40 q^{470}+640 q^{480}+484 q^{484}+616 q^{486}+656 q^{489}-500 q^{500}+O[q]^{501}$\\
\vspace{8pt}
$\Theta_{L_{\mathfrak{a}, P}}=2 q^2-4 q^3+4 q^7+8 q^8-10 q^{10}-16 q^{12}+20 q^{15}+14 q^{18}-44 q^{23}+8 q^{27}+16 q^{28}+32 q^{32}-20 q^{35}-40 q^{40}-32 q^{42}+76 q^{43}+4 q^{47}-64 q^{48}+50 q^{50}-44 q^{58}+80 q^{60}+28 q^{63}-116 q^{67}+56 q^{72}-100 q^{75}+124 q^{82}+76 q^{83}+88 q^{87}-70 q^{90}-176 q^{92}-66 q^{98}-44 q^{103}+124 q^{107}+32 q^{108}+64 q^{112}+220 q^{115}-116 q^{122}-248 q^{123}-236 q^{127}+128 q^{128}-40 q^{135}+352 q^{138}-80 q^{140}+132 q^{147}-160 q^{160}-190 q^{162}-164 q^{163}+244 q^{167}-128 q^{168}+304 q^{172}+100 q^{175}-284 q^{178}+232 q^{183}+16 q^{188}-256 q^{192}+200 q^{200}+244 q^{202}-88 q^{203}-308 q^{207}+160 q^{210}-380 q^{215}+76 q^{218}+436 q^{223}-356 q^{227}-176 q^{232}-20 q^{235}+320 q^{240}+242 q^{242}+308 q^{243}-250 q^{250}+112 q^{252}-608 q^{258}-284 q^{263}+568 q^{267}-464 q^{268}-32 q^{282}+316 q^{283}+248 q^{287}+224 q^{288}+220 q^{290}+556 q^{298}-400 q^{300}-488 q^{303}-596 q^{307}-140 q^{315}-352 q^{322}-152 q^{327}+496 q^{328}+304 q^{332}+580 q^{335}+338 q^{338}-328 q^{343}-116 q^{347}+352 q^{348}-280 q^{360}-716 q^{362}-484 q^{363}+724 q^{367}-704 q^{368}+500 q^{375}+64 q^{378}-44 q^{383}+532 q^{387}-264 q^{392}+928 q^{402}-620 q^{410}-176 q^{412}-380 q^{415}+28 q^{423}-232 q^{427}+496 q^{428}+128 q^{432}-440 q^{435}+796 q^{443}-1112 q^{447}+256 q^{448}+350 q^{450}-524 q^{458}+880 q^{460}-764 q^{463}+124 q^{467}+604 q^{482}+704 q^{483}+484 q^{487}-464 q^{488}+330 q^{490}-992 q^{492}-608 q^{498}+O[q]^{501}$

}
\end{center}

\begin{center}
{\footnotesize
$\Psi_{K, \Lambda}^{(1)}(z)=q+2 q^2-4 q^3+4 q^4-5 q^5-8 q^6+4 q^7+8 q^8+7 q^9-10 q^{10}-16 q^{12}+8 q^{14}+20 q^{15}+16 q^{16}+14 q^{18}-20 q^{20}-16 q^{21}-44 q^{23}-32 q^{24}+25 q^{25}+8 q^{27}+16 q^{28}-22 q^{29}+40 q^{30}+32 q^{32}-20 q^{35}+28 q^{36}-40 q^{40}+62 q^{41}-32 q^{42}+76 q^{43}-35 q^{45}-88 q^{46}+4 q^{47}-64 q^{48}-33 q^{49}+50 q^{50}+16 q^{54}+32 q^{56}-44 q^{58}+80 q^{60}-58 q^{61}+28 q^{63}+64 q^{64}-116 q^{67}+176 q^{69}-40 q^{70}+56 q^{72}-100 q^{75}-80 q^{80}-95 q^{81}+124 q^{82}+76 q^{83}-64 q^{84}+152 q^{86}+88 q^{87}-142 q^{89}-70 q^{90}-176 q^{92}+8 q^{94}-128 q^{96}-66 q^{98}+100 q^{100}+122 q^{101}-44 q^{103}+80 q^{105}+124 q^{107}+32 q^{108}+38 q^{109}+64 q^{112}+220 q^{115}-88 q^{116}+160 q^{120}+121 q^{121}-116 q^{122}-248 q^{123}-125 q^{125}+56 q^{126}-236 q^{127}+128 q^{128}-304 q^{129}-232 q^{134}-40 q^{135}+352 q^{138}-80 q^{140}-16 q^{141}+112 q^{144}+110 q^{145}+132 q^{147}+278 q^{149}-200 q^{150}-160 q^{160}-176 q^{161}-190 q^{162}-164 q^{163}+248 q^{164}+152 q^{166}+244 q^{167}-128 q^{168}+169 q^{169}+304 q^{172}+176 q^{174}+100 q^{175}-284 q^{178}-140 q^{180}-358 q^{181}+232 q^{183}-352 q^{184}+16 q^{188}+32 q^{189}-256 q^{192}-132 q^{196}+200 q^{200}+464 q^{201}+244 q^{202}-88 q^{203}-310 q^{205}-88 q^{206}-308 q^{207}+160 q^{210}+248 q^{214}-380 q^{215}+64 q^{216}+76 q^{218}+436 q^{223}+128 q^{224}+175 q^{225}-356 q^{227}-262 q^{229}+440 q^{230}-176 q^{232}-20 q^{235}+320 q^{240}+302 q^{241}+242 q^{242}+308 q^{243}-232 q^{244}+165 q^{245}-496 q^{246}-304 q^{249}-250 q^{250}+112 q^{252}-472 q^{254}+256 q^{256}-608 q^{258}-154 q^{261}-284 q^{263}+568 q^{267}-464 q^{268}+38 q^{269}-80 q^{270}+704 q^{276}-160 q^{280}-418 q^{281}-32 q^{282}+316 q^{283}+248 q^{287}+224 q^{288}+289 q^{289}+220 q^{290}+264 q^{294}+556 q^{298}-400 q^{300}+304 q^{301}-488 q^{303}+290 q^{305}-596 q^{307}+176 q^{309}-140 q^{315}-320 q^{320}-496 q^{321}-352 q^{322}-380 q^{324}-328 q^{326}-152 q^{327}+496 q^{328}+16 q^{329}+304 q^{332}+488 q^{334}+580 q^{335}-256 q^{336}+338 q^{338}-328 q^{343}+608 q^{344}-880 q^{345}-116 q^{347}+352 q^{348}-22 q^{349}+200 q^{350}-568 q^{356}-280 q^{360}+361 q^{361}-716 q^{362}-484 q^{363}+464 q^{366}+724 q^{367}-704 q^{368}+434 q^{369}+500 q^{375}+32 q^{376}+64 q^{378}+944 q^{381}-44 q^{383}-512 q^{384}+532 q^{387}-202 q^{389}-264 q^{392}+400 q^{400}-478 q^{401}+928 q^{402}+488 q^{404}+475 q^{405}-176 q^{406}-802 q^{409}-620 q^{410}-176 q^{412}-616 q^{414}-380 q^{415}+320 q^{420}-778 q^{421}+28 q^{423}-232 q^{427}+496 q^{428}-760 q^{430}+128 q^{432}-440 q^{435}+152 q^{436}-231 q^{441}+796 q^{443}+710 q^{445}+872 q^{446}-1112 q^{447}+256 q^{448}+398 q^{449}+350 q^{450}-712 q^{454}-524 q^{458}+880 q^{460}+842 q^{461}-764 q^{463}-352 q^{464}+124 q^{467}-464 q^{469}-40 q^{470}+640 q^{480}+604 q^{482}+704 q^{483}+484 q^{484}+616 q^{486}+484 q^{487}-464 q^{488}+656 q^{489}+330 q^{490}-992 q^{492}-608 q^{498}-500 q^{500}+O[q]^{501}$

}
\end{center}

\end{appendix}


\begin{thebibliography}{99}

\bibitem{Ahlgren}
S. Ahlgren, 
{Multiplicative Relations in Powers of Euler's Product}, 
{\sl Journal of Number Theory}, 
{\bf 89} (2001), 222--233. 


\bibitem{BKST}
E. Bannai, M. Koike, M. Shinohara, M. Tagami, 
{Spherical designs attached to extremal lattices and the modulo $p$ property of Fourier coefficients of extremal modular forms}, 
{\sl Mosc. Math. J. }, 
{\bf 6-2} (2006), 225--264. 

\bibitem{Toy-BM} E. Bannai, T. Miezaki,
{Toy models for D. H. Lehmer's conjecture}, 
accepted in {\sl J. Math. Soc. Japan}. 


\bibitem{Cox} D.\ A.\ Cox,
{\sl Primes of the form $x^2 + ny^2$. Fermat, class field theory and
complex multiplication.},
A Wiley-Interscience Publication. John Wiley
\& Sons, Inc., New York, 1989.



\bibitem{HP}
P. de la Harpe and C. Pache, 
Cubature formulas, geometrical designs, reproducing kernels, and Markov operators, {\it Infinite groups: geometric, combinatorial and dynamical aspects}, 
{\sl Progr. Math.}, Birkh\"auser, Basel, 
{\bf 248} (2005), 219--267. 

\bibitem{HPV}
P. de la Harpe, C. Pache, B. Venkov, 
{Construction of spherical cubature formulas using lattices}, 
{\sl Algebra i Analiz}, 
{\bf 18-1} (2006), 162--186, ; translation in 
{\sl St. Petersburg Math. J.} {\bf 18-1} (2007), 119--139. 

\bibitem{DGS}P. Delsarte, J.-M. Goethals, and J. J. Seidel, 
{Spherical codes and designs}, 
{\sl Geom. Dedicata}
{\bf 6} (1977), 363-388. 

\bibitem{Hecke} E. Hecke,
{\sl Mathematische Werke},
Vandenhoeck \& Ruprecht, G\"ottingen, 1983. 


\bibitem{Katz}
N. Katz, 
{An overview of Deligne's proof of the Riemann hypothesis for varieties over finite fields}, 
{\sl Amer. Math. Soc. Proc. Symp. Pure Math.}, 
{\bf 28} (1976), 275--305. 

\bibitem{Kob} N. Koblitz,
{\sl Introduction to Elliptic Curves and Modular Forms},
Springer-Verlag, Berlin/New York, 1984. 

\bibitem{Lehmer}D. H. Lehmer, 
{The vanishing of Ramanujan's $\tau (n)$}, 
{\sl Duke Math. J.}
{\bf 14} (1947), 429-433. 


\bibitem{miezaki}
{\sl T. Miezaki's website}, 
http://sites.google.com/site/tmiezaki/home/

\bibitem{Miyake}
T. Miyake, 
{\sl Modular forms}, 
Translated from the Japanese by Yoshitaka Maeda. 
Springer-Verlag, Berlin, 1989.

\bibitem{Web} K. Ono,
{\sl The web of modularity{\rm :} arithmetic of the coefficients of modular forms and q-series},
CBMS Regional Conference Series in Mathematics, vol. 102, Published for the Conference
Board of the Mathematical Sciences, Washington, DC, 2004. 

\bibitem{Pache}
C. Pache, 
{Shells of selfdual lattices viewed as spherical designs}, 
{\sl International Journal of Algebra and Computation}
{\bf 5} (2005), 1085--1127. 

\bibitem{Prime}
R. Prime, 
{Heche character}, 
http://www.math.uconn.edu/~prime/whatthehecke.pdf


\bibitem{Sage}
{\sl Sage}, 
http://www.sagemath.org/

\bibitem{Sch1}
B. Schoeneberg, 
{Das Verhalten von mehrfachen Thetareihen bei Modulsubstitutionen}, (German)
{\sl Math. Ann.}
{\bf 116-1} (1939), 511--523. 

\bibitem{Sch2} B. Schoeneberg, 
{\sl Elliptic modular functions: an introduction},
Translated from the German by J. R. Smart and E. A. Schwandt. Die Grundlehren der mathematischen Wissenschaften, Band 203. Springer-Verlag, New York-Heidelberg, 1974. 

\bibitem{Serre1}
J.-P. Serre, 
{\sl A course in arithmetic}, 
Translated from the French. Graduate Texts in Mathematics, {\bf 7}, 
Springer-Verlag, New York-Heidelberg, 1973. 

\bibitem{Serre}
J.-P. Serre, 
{Sur la lacunarit\'e des puissances de $\eta$}, 
{\sl Glasgow Math. J.} 
{\bf 27} (1985), 203--221.



\bibitem{Venkov1}
B. B. Venkov, 
{Even unimodular extremal lattices}, (Russian) 
{\sl Algebraic geometry and its applications. Trudy Mat. Inst. Steklov.}, 
{\bf 165} (1984), 43--48; 
translation in {\sl Proc. Steklov Inst. Math.} {\bf 165} (1985) 47--52. 

\bibitem{Venkov2}
B. B. Venkov, 
{R\'eseaux et designs sph\'eriques}, (French) [Lattices and spherical designs], 
{\sl R\'eseaux euclidiens, designs sph\'eriques et formes modulaires}, 
Monogr. Enseign. Math. 
{\bf 37} (2001), 10--86, Enseignement Math., Geneva.

\bibitem{Zagier}
D.\ B.\ Zagier, {\sl Zetafunktionen und quadratische K\"orper 
{\rm :} eine Einf\"uhrung in die h\"ohere Zahlentheorie}, 
Springer-Verlag, Berlin-New York, 1981. 
\end{thebibliography}
\end{document}